\documentclass[10pt,a4paper]{article}
\usepackage{geometry} % see geometry.pdf on how to lay out the page. There's lots.
\usepackage{amssymb,amsmath,amsthm}
\usepackage[T1]{fontenc}
\usepackage{times,color}
\usepackage{url,multirow}
\usepackage{graphicx}
\usepackage{epstopdf,setspace}
\allowdisplaybreaks

\def\de{{\rm d}}

\newtheorem{theorem}{Theorem}[section]

\newtheorem{lemma}{Lemma}[section]

\numberwithin{equation}{section}

\title{On a family of test statistics for discretely observed  diffusion processes}
\author{A. De Gregorio\footnote{
Dipartimento di Scienze Statistiche,
P.le Aldo Moro 5,
 00185- Rome, Italy - alessandro.degregorio@uniroma1.it}\quad
 S.M. Iacus\footnote{Dipartimento di Scienze Economiche, Aziendali e Statistiche,
 Via Conservatorio 22, 20122 - Milan, Italy - stefano.iacus@unimi.it}}

\begin{document}

\maketitle

\begin{abstract}
We consider parametric hypotheses testing for multidimensional  ergodic diffusion processes observed at discrete
time. We propose a  family of  test statistics, related to the so called $\phi$-divergence measures. By taking into account the quasi-likelihood approach developed for studying the stochastic differential equations, it is proved that the tests in this family are all asymptotically distribution free. In other words, our test statistics weakly converge to the chi squared distribution. Furthermore, our test statistic is compared with the quasi likelihood ratio test. In the case of contiguous alternatives, it is also possible to study in detail the power function of the tests.

Although all the tests in this family are asymptotically equivalent, we show by Monte Carlo analysis that, in the small sample case, the performance of the test strictly depends on the choice of the function $\phi$. Furthermore, in this framework, the simulations show that there are not uniformly most powerful tests.
\end{abstract}


\section{Introduction}
In the last years there has been a growing interest around  diffusion processes defined by means of stochastic differential equations. Indeed, diffusion models are useful for describing the random evolution of real phenomena studied, for instance, in physics and biology. These random processes have often been applied in mathematical finance and econometric theory to describe the behavior of stock prices, exchange rates, interest rates. Although, the underlying random model evolves continuously in time, the time data are always recorded at discrete instants (e.g. weekly, daily or each minute). For these reasons the inference problems for discretely observed diffusion processes have been tackled by several authors in order to provide useful statistical tools for the applied researchers and practitioners.

Let $X_t$, $t\in[0,T]$,  be a $d$-dimensional diffusion process solution of the following $d$-dimensional stochastic differential equation
$\de X_t = b(\alpha, X_t) \de t + \sigma(\beta,X_t)  \de W_t$, where
 functions $b$ and $\sigma$ are suitably regular and known up to  the parameters $\alpha\in \mathbb{R}^p$ and $\beta\in\mathbb{R}^q$. The process $X_t$ is discretely observed at times  $t_i$, such that $t_i-t_{i-1}=\Delta_n<\infty$ for $1\leq i\leq n $. In order to test the parametric vector $\theta=(\alpha,\beta)$ of the process $X_t$, $t\in[0,T]$, this paper proposes the construction of a family of test statistics for the following hypotheses testing problem
$$H_0:\theta=\theta_0\quad \text{versus}\quad H_1:\theta\neq \theta_0.$$

The problem of testing parametric hypotheses for diffusion processes is still a developing stream of research.
In continuos time, Kutoyants (2004) and Dachian and Kutoyants (2008) considered the problem  for ergodic diffusion models; Kutoyants (1984) considered the same problem for small diffusion processes. In discrete time, Lee and Wee (2008) dealt with a parametric version of the score marked empirical process test statistics, while
A\"it-Sahalia (1996), Giet and Lubrano (2008) and Chen {\it
et al.} (2008) proposed tests based on the several distances between parametric and nonparametric estimation of the invariant
density of ergodic diffusion processes.

The test statistics introduced in this paper has been inspired by the $\phi$-divergence theory that we recall briefly in the following. Let   $\{ p(X,\theta),
\theta\in\Theta\}$ be a family of probability densities. Denote by $E_\theta$  the expected value with respect to $ P_\theta$, the true law of
the observations $X$. Let $\phi:[0,\infty)\to \mathbb{R}$ be a convex and continuous function. Furthermore, its restriction on $(0,\infty)$ is finite,  two times continuously differentiable and such that $\phi(1)=\phi'(1)=0$ and $\phi''(1)=1$. Then the $\phi$-divergence measure between the two models $p(X,\theta)$ and $p(X,\theta_0)$, $\theta\neq\theta_0$, is defined as
\begin{equation}
\label{eq:div}
D_\phi(\theta,\theta_0)=E_{\theta_0}\phi\left(\frac{p(X,\theta)}{p(X,\theta_0)}\right)
\end{equation}

We remind that $\phi$-divergences are contain as special cases many divergences like the $\alpha$-divergences (Csisz\'ar, 1967) 
%and Amari,
%1985?), 
the  Kullback-Leibler and the  Hellinger divergences 
%(see, e.g., Beran, 1977?, Simpson, 1989?), 
the R\'enyi's divergence, 
%(R\'enyi, 1961)?, 
the power-divergences studied in Cressie
and Read (1984).
%Liese and Vajda (1987) provide extensive study of a modified version of $\alpha$-divergences, 

The $\phi$-divergence measures have been used for testing hypotheses in parametric models. The reader can consult on this point, for example Morales {\it et al.}  (1997) and Pardo (2006). Morales {\it et al.} (2004) applied modified R\'enyi's divergence for testing problems on the family of exponential models.

Given a sample of $n$ independent and identically distributed (i.i.d.) observations and some asymptotically efficient estimator $\hat\theta_n$, to test
 $H_0 : \theta=\theta_0$ against $H_1 : \theta\neq \theta_0$ the $\phi$-divergence test statistic is given by $D_\phi(\hat\theta_n,\theta_0)$.
For a one-dimensional diffusion process with $\beta=\beta^*$ assumed known, the $\phi$-divergence is formally given by
\begin{equation}\label{eq:divint}
D_\phi(\theta,\theta_0)=\int\phi\left(\frac{\de P_\theta}{\de P_{\theta_0}}\right)\de P_{\theta_0}
\end{equation}
where 

$$\frac{\de P_\theta}{\de P_{\theta_0}}=\exp\left\{\int_0^T\frac{ b(\alpha, X_t) - b(\alpha_0, X_t) }{\sigma^2(\beta^*,X_t) }\de X_t-\frac12\int_0^T\frac{b^2(\alpha, X_t) - b^2(\alpha_0, X_t)}{\sigma(\beta^*,X_t) }\de t\right\}.$$
The study of $\phi$-divergences for continuous time observations of diffusion processes has been considered in
 Vajda (1990). Explicit derivations of the R\'enyi information on the invariant law of  ergodic diffusion processes have been presented in De Gregorio and Iacus (2009). K\"uchler and S\o rensen  (1997) provide several results on the likelihood ratio test statistics  statistics for exponential families of diffusion processes.
For small diffusion processes, Uchida and Yoshida (2004)  derived information criteria using Malliavin calculus.
For discrete time observations, R\'enyi divergence measures has been considered in
Rivas {\it et al.} (2005).

Formula \eqref{eq:divint} is not useful for testing problems on discretely observed diffusion processes. Therefore, we take into account an alternative approach. Let us consider the following  statistic
$$\mathcal{D}_{\phi,n}(\theta,\theta_0)=\frac{1}{n}\sum_{i=1}^n\phi\left(\frac{p_i(\theta) }{p_i(\theta_0)}\right)
$$
where $p_i(\cdot)$ is a suitable (gaussian) approximation of the transition density of the process $X_t$ from $X_{t_{i-1}}$ to $X_{t_i}$.
Notice that the function $\mathcal{D}_{\phi,n}$ is not a true $\phi$-divergence, nor a proper approximation of it, but its behaviour is studied in Section 4.

Let $\hat\theta_n$ be any consistent estimator of $\theta$, then
the family of test statistics is defined as follows
$T_{\phi,n}(\hat\theta_n,\theta_0):=2n\mathcal{D}_{\phi,n}(\hat\theta_n,\theta_0)$.
By exploiting the quasi-likelihood approach developed by Genon-Catalot and Jacod (1993), Florens-Zmirou (1993), Kessler (1997) and Yoshida (1992, 2011), we derive the asymptotic distribution of the tests under the null hypothesis and under the case of contiguous alternatives.

The paper is organized as follows. In Section \ref{sec:model} we introduce the notations, the model, the regularity conditions and the asymptotic framework. Section \ref{sec:preliminary} contains preliminary results needed in Section \ref{sec:pseudo} where the family of test statistics are introduced
and studied. Section \ref{sec:numerics} is devoted to numerical analysis of the performance of the tests for small sample sizes. The methodology developed in this paper can be applied to other diffusion models, Section \ref{ext} discuss this point. The Appendix contains some auxiliary but useful results for the proofs of this presented in this work. The tables are collected at the end of the paper.

\section{Notation and basic assumptions}\label{sec:model}

Let $X_t$, $t\in[0,T]$,  be a $d$-dimensional diffusion process solution of the following multivariate stochastic differential equation
\begin{equation}\label{eq:sde}
\de X_t = b(\alpha, X_t) \de t + \sigma(\beta,X_t)  \de W_t,\quad X_0=x_0,
\end{equation}
 where $\alpha=(\alpha_1,...,\alpha_{p})'\in\Theta_p\subset \mathbb{R}^p$, $p\geq1$, $\beta=(\beta_1,...,\beta_q)'\in\Theta_q\subset \mathbb{R}^q$, $q\geq1$, are $p\times 1$ and $q\times 1$ vectors respectively, $b:\Theta_p\times\mathbb{R}^d\to \mathbb{R}^d$, $\sigma:\Theta_q\times \mathbb{R}^d\to \mathbb{R}^d\times \mathbb{R}^m$ and $\{W_t,0 \leq t\leq T\}$, is a standard Brownian motion in $\mathbb{R}^m$. We assume that the functions $b$ and $\sigma$ are known up to  the parameters $\alpha$ and $\beta$. 
 
The sample path of $X_t$ is observed only at $n+1$ equidistant discrete times $t_i$, such that $t_i-t_{i-1}=\Delta_n<\infty$ for $1\leq i\leq n $ (with $t_0=0$ and $t_{n}=T$). The asymptotic scheme adopted in this paper is the following: $T=n\Delta_n\to \infty$, $\Delta_n\to 0$ and $n\Delta_n^2\to 0$ as $n\to \infty$. The previous scheme is called rapidly increasing design, where the number of observations grows over the time but no so fast.

The following notations will be used throughout the rest of the paper:
\begin{itemize}
\item We denote by $\theta=(\alpha,\beta)\in\Theta_p\times \Theta_q=\Theta$ the $(p+q)\times 1$ parametric vector and with $\theta_0=(\alpha_0,\beta_0)$ its unknown true value. The parameter space $\Theta$ is a compact set of $ \mathbb{R}^{p+q}$. 

\item ${\bf X}_n=\{X_{t_i}\}_{0\leq i\leq n}$ represents our random sample with values in $\mathbb{R}^{(n+1)\times d}$.

\item For a matrix $A$, we denote by $A^{-1}$ the inverse of $A$ and by $|A|^2=\text{tr}(AA')$, i.e. the sum of squares of the elements on the diagonal of matrix $A$. 
\item We set $\Sigma(\beta,x)=\sigma(\beta,x)\sigma(\beta,x)'$, $\Xi(\beta,x)=\Sigma^{-1}(\beta,x)$ and $\overline{X}_i(\alpha)=X_{t_i}-X_{t_{i-1}}-\Delta_n b(\alpha,X_{t_i})$. 

\item For $\mu=(\mu_1,...,\mu_m)$, $\partial_{\mu_k}:=\frac{\partial}{\partial \mu_k}$,  $\partial^2_{\mu_k}:=\frac{\partial^2}{\partial\mu_k^2}$,  $\partial_{\mu_k\mu_k'}^2:=\frac{\partial^2}{\partial \mu_k\partial\mu_k'}$,  $\partial_\mu=(\partial_{\mu_1},...,\partial_{\mu_m})'$ and $\partial_\mu^2=[\partial_{\mu_k\mu_k'}^2]_{k,k'=1,...,m}$ denotes the Hessian matrix.

\item If $f:\Theta\times\mathbb{R}^d\to \mathbb{R}$, we denote by $f_i(\theta)$ the value $f(\theta, X_{t_i})$; for example $\Sigma_i(\beta)=\Sigma(\beta,X_{t_i})$. Furthermore, if $f$ is a tensor, we indicate with the upper index its components: when $f$ is a matrix $f^{l,m}$ represents its $(l,m)$-component.

\item  For $0\leq i\leq n$, $t_i=i\Delta_n$, $\mathcal{G}_i^n=\sigma(W_s,s\leq t_i)$. 

\item Let $u_n$ be a $\mathbb{R}$-valued sequence. We indicate by $R$ a function $\Theta\times\mathbb{R}^d\to \mathbb{R}$ for which there exists a constant $C$ such that
$$R(\theta,u_n,x)\leq u_nC(1+|x|)^C,\quad \theta\in\Theta, x\in \mathbb{R}^d, n\in\mathbb{N}.$$

\end{itemize}

We need  some assumptions on the regularity of the process $X_t,t \in[0,T]$:
\begin{itemize}
\item[$\mathcal A_1.$]  there exists a constant $C$ such that
$$|b(\alpha_0,x)-b(\alpha_0,y)|+|\sigma(\beta_0,x)-\sigma(\beta_0,y)|\leq C|x-y|;$$

\item[$\mathcal A_2.$] $\inf_{\beta,x}\det(\Sigma(\beta,x))>0$;

\item[$\mathcal A_3.$] the process $X_t,t \in[0,T],$ is ergodic for $\theta=\theta_0$ with
invariant probability measure $\mu_\theta$. Thus
$$\frac1T\int_0^Tf(X_t)\de t\stackrel{P_\theta}{\to}\int f(x)\mu_\theta(\de x)$$
as $T\to\infty$, where $f$ is a $\mu_\theta$-integrable function. 

\item[$\mathcal A_4.$] if the coefficients $b(\alpha,x)=b(\alpha_0,x)$ and
$\sigma(\beta,x)=\sigma(\beta_0,x)$  for all
$x$  ($\mu_{\theta_0}$-almost surely), then $\alpha=\alpha_0$ and $\beta=\beta_0$;

\item[$\mathcal A_5.$]  for all $m\geq 0$ and for all $\theta\in\Theta$, $\sup_t
E|X_t|^m<\infty$;

\item[$\mathcal A_6.$] for every $\theta\in\Theta$, the coefficients $b(\alpha,x)$ and
$\sigma(\beta,x)$ are five times continuously differentiable with respect to $x$ and
the derivatives are bounded by a polynomial function in $x$, uniformly in
$\theta$;

\item[$\mathcal A_7.$]  the coefficients $b(\alpha,x)$ and $\sigma(\beta,x)$
and all their partial derivatives with respect to $x$ up to order 2 are
three times continuously differentiable with respect to $\theta$ for all $x$ in the
state space. All derivatives with respect to $\theta$ are bounded by a polynomial
function in $x$, uniformly in $\theta$.
\end{itemize}

We observe that the assumption $\mathcal{A}_1$ ensures the existence and uniqueness of a solution to \eqref{eq:sde} for the value $\theta_0=(\alpha_0,\beta_0)$ of $\theta\in\Theta$, while $\mathcal A_4$ is the identifiability condition.
Hereafter, we assume that the conditions $\mathcal{A}_1-\mathcal{A}_7 $ hold. These conditions are equivalent to the ones in Uchida and Yoshida (2005)  and Kessler (1997) for what concerns the regularity of the model.

\section{Preliminary results}\label{sec:preliminary}

Since the transition density between $X_{t_{i-1}}$ and $X_{t_i}$ is almost always unknown, for the estimation of stochastic differential equations it has been developed an alternative tool with respect to the likelihood function. Therefore, in order to introduce the test statistics for the stochastic differential equation \eqref{eq:sde}, we consider a quasi-likelihood approach based on a suitable contrast function. In other words, we deal with the negative quasi-loglikelihood function $\mathbb{H}_n:\mathbb{R}^{(n+1)\times d}\times \Theta\to \mathbb{R}$ 
\begin{align}\label{qlik}
\mathbb{H}_n({\bf X}_n,\theta)
:=\sum_{i=1}^n \mathbb{H}_i(\theta)
:=\frac12\sum_{i=1}^n\left\{\log\det(\Sigma_{i-1}(\beta))
+\frac{
1}{\Delta_n}\overline{X}_i'(\alpha)\Xi_{i-1}(\beta)\overline{X}_i(\alpha)\right\}.
\end{align}
The function \eqref{qlik} is obtained by discretization of the continuous time stochastic differential equation \eqref{eq:sde} by Euler-Maruyama scheme, that is
\begin{align*}
X_{t_i}-X_{t_{i-1}}&=\int_{t_{i-1}}^{t_i}b(\alpha, X_s)\de s+\int_{t_{i-1}}^{t_i}\sigma(\beta,X_s)\de W_s\\
&\cong b(\alpha, X_{t_{i-1}})\Delta_n+\sigma(\beta,X_{t_{i-1}})(W_{t_i}-W_{t_{i-1}})
\end{align*}
and the increments $(X_{t_i}-X_{t_{i-1}})$ are (approximately) conditionally independent Gaussian random variables for $i=1,...,n$.

The quasi-likelihood \eqref{qlik} has been used by, e.g.,  Florens-Zmirou (1993), Yoshida (1992, 2011), Genon-Catalot and Jacod (1993)   and Kessler (1997) to make inference for stochastic differential equations. This last author considered a more general asymptotic scheme, that is $n\Delta_n^p\to 0,p\geq 2,$ and generalizes the contrast function \eqref{qlik} improving the convergence results. For the sake of simplicity we focus our attention to the case $n\Delta_n^2\to 0$. 

Let $\hat\theta_n:\mathbb{R}^{(n+1)\times d}\to \Theta$ be the quasi-maximum likelihood estimator of $\theta\in\Theta$, based on \eqref{qlik}, that is
$$
\hat{\theta}_n=(\hat\alpha_n,\hat\beta_n)=\arg\min_\theta \mathbb{H}_n({\bf X}_n,\theta).
$$ 
and let us consider the matrix
$$\varphi(n)=\left(%
\begin{array}{cc}
  \frac{1}{n{\Delta_n}}{\bf I}_p& 0 \\
  0 &  \frac{1}{n}{\bf I}_q\\
\end{array}%
\right)$$
where ${\bf I}_p$ and ${\bf I}_q$ are respectively the identity matrix of order $p$ and $q$.
The regularity conditions $\mathcal A_1 - \mathcal A_7$ imply some fundamental results which have a crucial role for analyzing the asymptotic distributional behavior of the estimators (and thus of our test statistics as we will show in the next Section). Indeed, as shown in Kessler (1997) and Yoshida (2011),
$\hat\theta_n$ is  a consistent  estimator of $\theta_0$ and asymptotically Gaussian with rate of convergence given by $\varphi(n)^{-1/2}$, i.e.
\begin{equation}\label{eq:conest}
\varphi(n)^{-1/2}(\hat\theta_n-\theta_0)\stackrel{d}{\to}N(0,\mathcal{I}(\theta_0)^{-1}).
\end{equation}
where 
$\mathcal I(\theta_0)$ is the positive definite and invertible
Fisher information matrix at $\theta_0$  given by
$$\mathcal I(\theta_0)=\left(%
\begin{array}{cc}
 [\mathcal I_b^{j,k}(\theta_0)]_{j,k=1,...,p} & 0 \\
  0 & [\mathcal I_\sigma^{j,k}(\theta_0)]_{j,k=1,...,q} \\
\end{array}%
\right)$$ where
\begin{align*}
&\mathcal I_b^{j,k}(\theta_0)=\int\left(\partial_{\alpha_j} b(\alpha_0,x)\right)'\Xi(\beta_0,x)\partial_{\alpha_k} b(\alpha_0,x)
\mu_{\theta_0}(\de x)\,,\\
&\mathcal I_\sigma^{j,k}(\theta_0)=\frac12\int\text{tr}\left[\partial_{\beta_j} \Sigma(\beta_0,x)\Xi(\beta_0,x)\partial_{\beta_k} \Sigma(\beta_0,x)\Xi(\beta_0,x)\right]
\mu_{\theta_0}(	\de x)\,.
\end{align*}
The matrix $\varphi(n)$ plays the role of the rate of convergence in the estimation problem for the stochastic differential equation \eqref{eq:sde}.

The Bayes type estimator $\tilde\theta_n=(\tilde\alpha_n,\tilde\beta_n)$ for $\theta$ is defined by 
\begin{align*}
\tilde\beta_n&=\left\{\int\exp\left(\mathbb{H}_n({\bf X}_n,(\alpha^*,\beta))\right)\pi_1(\de\beta)\right\}^{-1}\times\int\beta\exp\left(\mathbb{H}_n({\bf X}_n,(\alpha^*,\beta))\right)\pi_1(\de\beta)\\
\tilde\alpha_n&=\left\{\int\exp\left(\mathbb{H}_n({\bf X}_n,(\alpha,\tilde\beta_n))\right)\pi_2(\de\alpha)\right\}^{-1}\times\int\alpha\exp\left(\mathbb{H}_n({\bf X}_n,(\alpha,\tilde\beta_n))\right)\pi_2(\de\alpha)
\end{align*}
where $\alpha^*$ is an arbitrary constant and $\pi_1(\beta)$ and $\pi_2(\alpha)$ are the prior distributions of $\beta$ and $\alpha$ respectively. Yoshida (2011) proved that $\tilde\theta_n$ is a consistent estimator for $\theta_0$ and that the weak convergence \eqref{eq:conest} still holds for this estimator. 
Other classes of estimators with similar asymptotic properties exist in the literature, for a recent review see S\o rensen (2004).

The next result on the score function will be useful in the study of the asymptotic behavior of the  test statistics.

\begin{lemma}\label{lem1}
Let  $\Lambda_n(\theta)$ be a $(p+q)\times(p+q)$ matrix with $(j,k)$-component given by 

$$\Lambda_n^{j,k}(\theta)=\sum_{i=1}^n\partial_{\theta_j}\mathbb{H}_i(\theta)\partial_{\theta_k}\mathbb{H}_i( \theta),\quad j,k=1,2,...,p+q.$$
Under the conditions $\mathcal A_1 - \mathcal A_7$, the following property holds true
\begin{equation}\label{eq:conqlik}
 \varphi(n)^{1/2}\Lambda_n(\theta_0)  \varphi(n)^{1/2}\stackrel{P_{\theta_0}}{\to}\mathcal{I}(\theta_0).
\end{equation}
\end{lemma}

\begin{proof}
For the sake of simplicity we denote by
$\overline{X}_i:=\overline{X}_i(\alpha)$, $b_{i-1}:=b_{i-1}(\alpha)$, $\Xi_{i-1}:=\Xi_{i-1}(
 \beta)$. In order to prove 	\eqref{eq:conqlik}, we deal with the following expressions
\begin{align}
&\partial_{\alpha_j}\mathbb{H}_i(\theta)=\sum_{l,m=1}^d\partial_{\alpha_j}b_{i-1}^{l}\Xi_{i-1}^{l,m}\overline{X}_i^{m},\quad j=1,...,p,
\\
&\partial_{\beta_k}\mathbb{H}_i(\theta)=\frac12\left\{\sum_{l,m=1}^d\frac{\partial_{\beta_k}\Xi_{i-1}^{l,m}}{\Delta_n}\overline{X}_i^{m}\overline{X}_i^{l}+\frac{\partial_{\beta_k}\det(\Sigma_{i-1})}{\det(\Sigma_{i-1})}\right\},\quad k=1,...,q,
 \end{align}
and show that
 \begin{align}
&\frac{1}{n\Delta_n}\sum_{i=1}^n \partial_{\alpha_j}\mathbb{H}_i(\theta_0)\partial_{\alpha_k}\mathbb{H}_i(\theta_0)\stackrel{P_{\theta_0}}{\to}\mathcal I_b^{j,k}(\theta_0)\label{eq:condrift}\\
&\frac{1}{n}\sum_{i=1}^n \partial_{\beta_j}\mathbb{H}_i(\theta_0)\partial_{\beta_k}\mathbb{H}_i(\theta_0)\stackrel{P_{\theta_0}}{\to}\mathcal I_\sigma^{j,k}(\theta_0)\label{eq:condsigma}\\
&\frac{1}{n\sqrt{\Delta_n}}\sum_{i=1}^n\partial_{\alpha_j}\mathbb{H}_i(\theta_0)\partial_{\beta_k}\mathbb{H}_i(\theta_0)\stackrel{P_{\theta_0}}{\to}0\label{eq:condcross}
 \end{align}
 by means of Lemma \ref{lemmaGJ}.

Let us start proving the result \eqref{eq:condrift}. Bearing in mind Lemma \ref{lem0}-\ref{lemmaer}, we get that
 \begin{align}\label{eq:condrift1}
 &\frac{1}{n\Delta_n}\sum_{i=1}^n\left.E_{\theta_0}\{ \partial_{\alpha_j}\mathbb{H}_i(\theta_0)\partial_{\alpha_k}\mathbb{H}_i(\theta_0)\right|\mathcal{G}_{i-1}^n\}\\
& =\frac{1}{n\Delta_n}\sum_{i=1}^n\left.E_{\theta_0}\left\{\left(\sum_{l,m=1}^d\partial_{\alpha_j}b_{i-1}^{l}\Xi_{i-1}^{l,m}\overline{X}_i^{m}\right)\left(\sum_{l',m'=1}^d\partial_{\alpha_k}b_{i-1}^{l'}\Xi_{i-1}^{l',m'}\overline{X}_i^{m'}\right)\right|\mathcal{G}_{i-1}^n\right\}\notag\\
 &=\frac{1}{n\Delta_n}\sum_{i=1}^n\sum_{l,m=1}^d\sum_{l',m'=1}^d  \Xi_{i-1}^{l,m} \Xi_{i-1}^{l',m'}\partial_{\alpha_j}b_{i-1}^{l}\partial_{\alpha_k}b_{i-1}^{l'}\left.E_{\theta_0}\left\{\overline{X}_i^{m}\overline{X}_i^{m'}\right|\mathcal{G}_{i-1}^n\right\}\notag\\
&=\frac{1}{n}\sum_{i=1}^n\sum_{l,m=1}^d\sum_{l',m'=1}^d  \Xi_{i-1}^{l,m} \Xi_{i-1}^{l',m'}\Sigma_{i-1}^{m,m'}\partial_{\alpha_j}b_{i-1}^{l}\partial_{\alpha_k}b_{i-1}^{l'}+\frac1n\sum_{i=1}^nR(\theta_0,\Delta_n,X_{i-1})
\notag\\
&\stackrel{P_{\theta_0}}{\to}\sum_{l,m=1}^d\sum_{l',m'=1}^d \int \Xi^{l,m} \Xi^{l',m'}\Sigma^{m,m'}\partial_{\alpha_j}b^{l}\partial_{\alpha_k}b^{l'}\mu_{\theta_0}(\de x)=\mathcal I_b^{j,k}(\theta_0)
\notag
 \end{align}
and
  \begin{align}\label{eq:condrift2}
 &\frac{1}{n^2\Delta_n^2}\sum_{i=1}^nE_{\theta_0}\{(\partial_{\alpha_j}\mathbb{H}_i(\theta_0)\partial_{\alpha_k}\mathbb{H}_i(\theta_0))^2|\mathcal{G}_{i-1}^n\} \notag\\
 &=\frac{1}{n^2\Delta_n^2}\sum_{i=1}^n\left.E_{\theta_0}\left\{\left(\sum_{l,m=1}^d \sum_{l',m'=1}^d  \Xi_{i-1}^{l,m} \Xi_{i-1}^{l',m'}\partial_{\alpha_j}b_{i-1}^{l}\partial_{\alpha_k}b_{i-1}^{l'}\overline{X}_i^{m}\overline{X}_i^{m'}\right)^2\right|\mathcal{G}_{i-1}^n\right\}\notag\\
& \leq \frac{1}{n^2\Delta_n^2}\sum_{i=1}^n\sum_{l,m=1}^d\sum_{l',m'=1}^d  (\Xi_{i-1}^{l,m} \Xi_{i-1}^{l',m'}\partial_{\alpha_j}b_{i-1}^{l}\partial_{\alpha_k}b_{i-1}^{l'})^2\left.E_{\theta_0}\left\{(\overline{X}_i^{m})^2(\overline{X}_i^{m'})^2\right|\mathcal{G}_{i-1}^n\right\}\notag\\
&=\frac{1}{n^2}\sum_{i=1}^n\sum_{l,m=1}^d\sum_{l',m'=1}^d \left[ \left(\Xi_{i-1}^{l,m} \Xi_{i-1}^{l',m'}\partial_{\alpha_j}b_{i-1}^{l}\partial_{\alpha_k}b_{i-1}^{l'}\right)^2\left(\Sigma_{i-1}^{m,m}\Sigma_{i-1}^{m',m'}+2\left(\Sigma_{i-1}^{m,m'}\right)^2+R(\theta_0,\Delta_n^3,X_{i-1})\right)\right]\notag\\
&\stackrel{P_{\theta_0}}{\to}0
\end{align}

The proof of the result \eqref{eq:condsigma} is developed as follows
\begin{align*}
&\frac{1}{n}\sum_{i=1}^nE_{\theta_0}\{ \partial_{\beta_j}\mathbb{H}_i(\theta_0)\partial_{\beta_k}\mathbb{H}_i(\theta_0)|\mathcal{G}_{i-1}^n\}\notag\\
&=\frac{1}{4n}\sum_{i=1}^nE_{\theta_0}\Bigg\{\left[\sum_{l,m=1}^d\frac{\partial_{\beta_j}\Xi_{i-1}^{l,m}}{\Delta_n}\overline{X}_i^{m}\overline{X}_i^{l}+\frac{\partial_{\beta_j}\det(\Sigma_{i-1})}{\det(\Sigma_{i-1})}\right]\left[\sum_{l',m'=1}^d\frac{\partial_{\beta_k}\Xi_{i-1}^{l',m'}}{\Delta_n}\overline{X}_i^{m'}\overline{X}_i^{l'}+\frac{\partial_{\beta_k}\det(\Sigma_{i-1})}{\det(\Sigma_{i-1})}\right] |\mathcal{G}_{i-1}^n\Bigg\}\notag\\
&=\frac{1}{4n}\sum_{i=1}^n\Bigg\{\sum_{l,m=1}^d\sum_{l',m'=1}^d\partial_{\beta_j}\Xi_{i-1}^{l,m}\partial_{\beta_k}\Xi_{i-1}^{l',m'}(\Sigma_{i-1}^{l,m}\Sigma_{i-1}^{l',m'}+\Sigma_{i-1}^{m,m'}\Sigma_{i-1}^{l,l'}+\Sigma_{i-1}^{m,l'}\Sigma_{i-1}^{m',l})+R(\theta,\Delta_n^3,X_{i-1})\notag\\
&\quad+\frac{\partial_{\beta_k}\det(\Sigma_{i-1})}{\det(\Sigma_{i-1})}\sum_{l,m=1}^d\partial_{\beta_j}\Xi_{i-1}^{l,m}\Sigma_{i-1}^{l,m}+\frac{\partial_{\beta_j}\det(\Sigma_{i-1})}{\det(\Sigma_{i-1})}\sum_{l',m'=1}^d\partial_{\beta_k}\Xi_{i-1}^{l',m'}\Sigma_{i-1}^{l',m'}+R(\theta_0,\Delta_n^2,X_{i-1})\notag\\
&\quad+\frac{\partial_{\beta_j}\det(\Sigma_{i-1})}{\det(\Sigma_{i-1})}\frac{\partial_{\beta_k}\det(\Sigma_{i-1})}{\det(\Sigma_{i-1})}\Bigg\}\\
&=\frac{1}{2n}\sum_{i=1}^n\sum_{l,m=1}^d\sum_{l',m'=1}^d\partial_{\beta_j}\Sigma_{i-1}^{l,m}\partial_{\beta_k}\Sigma_{i-1}^{l',m'}\Xi_{i-1}^{l,m}\Xi_{i-1}^{l',m'}+\frac{1}{4n}\sum_{i=1}^n(R(\theta_0,\Delta_n^3,X_{i-1})+R(\theta_0,\Delta_n^2,X_{i-1}))\stackrel{P_{\theta_0}}{\to}\mathcal I_\sigma^{j,k}(\theta_0)
\end{align*}
where in the last step we have used the following relationship
$$\frac{\partial_{\beta_k}\det(\Sigma_{i-1})}{\det(\Sigma_{i-1})}=-\text{tr}[\partial_{\beta_k}\Xi_{i-1}\Sigma_{i-1}]=\text{tr}[\partial_{\beta_k}\Sigma_{i-1}\Xi_{i-1}]$$
and, once again, Lemma \ref{lem0}-\ref{lemmaer}.

In order to prove that
\begin{align*}
&\frac{1}{n^2}\sum_{i=1}^nE_{\theta_0}\{ (\partial_{\beta_j}\mathbb{H}_i(\theta_0)\partial_{\beta_k}\mathbb{H}_i(\theta_0))^2|\mathcal{G}_{i-1}^n\}\stackrel{P_{\theta_0}}{\to}0
\end{align*}
we use the same arguments in \eqref{eq:condrift2} and further note that 
\begin{align}
E_{\theta_0}\left\{\prod_{j=1}^4(\overline{X}_i^{k_j})^2|\mathcal{G}_{i-1}^n\right\}\leq E_{\theta_0}\left\{C(|X_i-X_{i-1}|^8+|\Delta_nb_{i-1}(\alpha_0)|^8)|\mathcal{G}_{i-1}^n\right\}=R(\theta_0,\Delta_n^4,X_{i-1})\
\end{align}
which is a consequence of Lemma \ref{lemmagr}.
By using the previous arguments, we get that
\begin{align*}
&\frac{1}{n\sqrt{\Delta_n}}\sum_{i=1}^nE_{\theta_0}\{ \partial_{\alpha_j}\mathbb{H}_i(\theta_0)\partial_{\beta_k}\mathbb{H}_i(\theta_0)|\mathcal{G}_{i-1}^n\}\\
&=\frac{1}{2n\sqrt{\Delta_n}}\sum_{i=1}^n\left.E_{\theta_0}\left\{\left(\sum_{l,m=1}^d\partial_{\alpha_j}b_{i-1}^{l}\Xi_{i-1}^{l,m}\overline{X}_i^{m}\right)\left(\sum_{l',m'=1}^d\frac{\partial_{\beta_k}\Xi_{i-1}^{l',m'}}{\Delta_n}\overline{X}_i^{m'}\overline{X}_i^{l'}+\frac{\partial_{\beta_k}\det(\Sigma_{i-1})}{\det(\Sigma_{i-1})}\right)\right|\mathcal{G}_{i-1}^n\right\}\\
&=\frac{1}{2n\sqrt{\Delta_n}}\sum_{i=1}^n\Bigg\{\sum_{l,m,l',m'=1}^d\frac{\partial_{\alpha_j}b_{i-1}^{l}\Xi_{i-1}^{l,m}}{\Delta_n}\left.E_{\theta_0}\left\{\overline{X}_i^{m}\overline{X}_i^{m'}\overline{X}_i^{l'}\right|\mathcal{G}_{i-1}^n\right\}\\
&\quad+\frac{\partial_{\beta_k}\det(\Sigma_{i-1})}{\det(\Sigma_{i-1})}\sum_{l,m=1}^d\partial_{\alpha_j}b_{i-1}^{l}\Xi_{i-1}^{l,m}\left.E_{\theta_0}\left\{\overline{X}_i^{m}\right|\mathcal{G}_{i-1}^n\right\}\Bigg\}\\
&=\frac{1}{2n}\sum_{i=1}^n[R(\theta_0,\sqrt{\Delta_n},X_{i-1})+R(\theta_0,\Delta_n^{3/2},X_{i-1})]\stackrel{P_{\theta_0}}{\to}0
\end{align*}
and
\begin{align*}
&\frac{1}{n^2\Delta_n}\sum_{i=1}^nE_{\theta_0}\{ (\partial_{\alpha_j}\mathbb{H}_i(\theta_0)\partial_{\beta_k}\mathbb{H}_i(\theta_0))^2|\mathcal{G}_{i-1}^n\}\stackrel{P_{\theta_0}}{\to}0
\end{align*}

This last step concludes the proof.

\end{proof}
\section{The family of test statistics}\label{sec:pseudo}

Let us remind that $\phi:[0,\infty)\to \mathbb{R}$ is convex and continuous. Furthermore, its restriction on $(0,\infty)$ is finite,  two times continuously differentiable.
The goal of this section is to construct a family of test statistics for the following hypotheses testing problem
$$H_0:\theta=\theta_0,\quad \text{vs} \quad H_1:\theta\neq \theta_0$$
concerning the stochastic differential equation \eqref{eq:sde}.
To this aim, let us consider the following quantity
\begin{equation}
\mathcal{D}_{\phi,n}(\theta,\theta_0):=\frac{1}{n}\sum_{i=1}^n\phi\left(\frac{p_i(\theta)}{p_i(\theta_0)}\right)
\end{equation}
where $p_i(\theta):=\exp(\mathbb{H}_i(\theta))$
 and $\mathbb{H}_i(\theta)$ is defined as in \eqref{qlik}. The statistic $\mathcal{D}_{\phi,n}(\theta,\theta_0)$ represents the empirical mean of the functions $\phi\left(\frac{p_i(\theta)}{p_i(\theta_0)}\right)$ which measure the discrepancy between two (approximated) parametric models given the sample ${\bf X}_n$.
 The statistics $\mathcal{D}_{\phi,n}(\theta,\theta_0)$ is not an approximation of the $\phi$-divergence \eqref{eq:div}
 because it does not converge to 
 $$\int\phi\left(\frac{\mu_\theta(x)}{\mu_{\theta_0}(x)}\right)\mu_{\theta_0}(x)\de x,$$
 but it proves to be useful in the construction of a new class of asymptotically distribution free  test statistics  as we will see in what follows.
The next result establishes  the convergence in probability of $\mathcal{D}_{\phi,n}(\theta,\theta_0)$.

 \begin{theorem}\label{conphi} 
 %Let $\phi:[0,+\infty) \to \mathbb R$ be a twice continuously differentiable function.
 Under the conditions $\mathcal A_1 - \mathcal A_7$, we have that
 \begin{equation}
 \mathcal{D}_{\phi,n}(\theta,\theta_0)\stackrel{P_{\theta_0}}{\to}U_\phi(\beta,\beta_0)
 \end{equation}
uniformly in $\theta$, where
\begin{align}
 U_\phi(\beta,\beta_0)&:=\int\Bigg\{\phi\left(\left(\frac{\det(\Sigma(\beta,x))}{\det(\Sigma(\beta_0,x))}\right)^{\frac12}\right)+\frac12\Bigg[\phi'\left(\left(\frac{\det(\Sigma(\beta,x))}{\det(\Sigma(\beta_0,x))}\right)^{\frac12}\right)\notag\\
 &\quad\times\left(\frac{\det(\Sigma(\beta,x))}{\det(\Sigma(\beta_0,x))}\right)^{\frac12}(\mathrm{tr}(\Xi(\beta,x)\Sigma(\beta_0,x))-d)\Bigg]\Bigg\}\mu_{\theta_0}(\de x)
 \end{align}
 
 \end{theorem}
 
\begin{proof}
Let us define
\begin{align*}
F(X_i)&=\phi\left(\frac{p_i(\theta)}{p_i(\theta_0)}\right)\\
&=\phi\left(\left(\frac{\det(\Sigma_{i-1}(\beta))}{\det(\Sigma_{i-1}(\beta_0))}\right)^{\frac12}
\exp\left\{\frac{
1}{2\Delta_n}\left(\overline{X}_i'(\alpha)\Xi_{i-1}(\beta)\overline{X}_i(\alpha)-\overline{X}_i'(\alpha_0)\Xi_{i-1}(\beta_0)\overline{X}_i(\alpha_0)\right)\right\}\right)\\
&=\phi\left(\left(\frac{\det(\Sigma_{i-1}(\beta))}{\det(\Sigma_{i-1}(\beta_0))}\right)^{\frac12}
\exp\left\{\frac{
1}{2\Delta_n}\sum_{l,m=1}^d\left(\overline{X}_i^l(\alpha)\Xi_{i-1}^{l,m}(\beta)\overline{X}_i^m(\alpha)-\overline{X}_i^l(\alpha_0)\Xi_{i-1}^{l.m}(\beta_0)\overline{X}_i^m(\alpha_0)\right)\right\}\right)
\end{align*}
By setting $\overline{x}(\alpha)=(x-X_{i-1}-\Delta_nb_{i-1}(\alpha))$, we observe that
\begin{align*}
&\partial_{x_k}F(x)\\
&=\phi'\left(\left(\frac{\det(\Sigma_{i-1}(\beta))}{\det(\Sigma_{i-1}(\beta_0))}\right)^{\frac12}
\exp{\left\{\frac{
1}{2\Delta_n}\sum_{l,m=1}^d\left(\overline{x}^l(\alpha)\Xi_{i-1}^{l,m}(\beta)\overline{x}^m(\alpha)-\overline{x}^l(\alpha_0)\Xi_{i-1}^{l,m}(\beta_0)\overline{x}^m(\alpha_0)\right)\right\}}\right)\\
&\quad\times \left(\frac{\det(\Sigma_{i-1}(\beta))}{\det(\Sigma_{i-1}(\beta_0))}\right)^{\frac12}
\exp\left\{\frac{
1}{2\Delta_n}\sum_{l,m=1}^d\left(\overline{x}^l(\alpha)\Xi_{i-1}^{l,m}(\beta)\overline{x}^m(\alpha)-\overline{x}^l(\alpha_0)\Xi_{i-1}^{l,m}(\beta_0)\overline{x}^m(\alpha_0)\right)\right\}\\
&\quad \times\frac{1}{\Delta_n}\sum_{m=1}^d\left(\Xi_{i-1}^{k,m}(\beta)\overline{x}^m(\alpha)-\Xi_{i-1}^{k,m}(\beta_0)\overline{x}^m(\alpha_0)\right)
\end{align*}
and
\begin{align*}
&\partial_{x_kx_{k'}}F(x)\\
&=\phi''\left(\left(\frac{\det(\Sigma_{i-1}(\beta))}{\det(\Sigma_{i-1}(\beta_0))}\right)^{\frac12}
\exp{\left\{\frac{
1}{2\Delta_n}\sum_{l,m=1}^d\left(\overline{x}^l(\alpha)\Xi_{i-1}^{l,m}(\beta)\overline{x}^m(\alpha)-\overline{x}^l(\alpha_0)\Xi_{i-1}^{l,m}(\beta_0)\overline{x}^m(\alpha_0)\right)\right\}}\right)\\
&\quad\times \left(\frac{\det(\Sigma_{i-1}(\beta))}{\det(\Sigma_{i-1}(\beta_0))}\right)
\exp\left\{\frac{
1}{\Delta_n}\sum_{l,m=1}^d\left(\overline{x}^l(\alpha)\Xi_{i-1}^{l,m}(\beta)\overline{x}^m(\alpha)-\overline{x}^l(\alpha_0)\Xi_{i-1}^{l,m}(\beta_0)\overline{x}^m(\alpha_0)\right)\right\}\\
&\quad \times\left[\frac{1}{\Delta_n}\sum_{m=1}^d\left(\Xi_{i-1}^{k,m}(\beta)\overline{x}^m(\alpha)-\Xi_{i-1}^{k,m}(\beta_0)\overline{x}^m(\alpha_0)\right)\right]^2\\
&\quad+\phi'\left(\left(\frac{\det(\Sigma_{i-1}(\beta))}{\det(\Sigma_{i-1}(\beta_0))}\right)^{\frac12}
\exp{\left\{\frac{
1}{2\Delta_n}\sum_{l,m=1}^d\left(\overline{x}^l(\alpha)\Xi_{i-1}^{l,m}(\beta)\overline{x}^m(\alpha)-\overline{x}^l(\alpha_0)\Xi_{i-1}^{l,m}(\beta_0)\overline{x}^m(\alpha_0)\right)\right\}}\right)\\
&\quad\times \left(\frac{\det(\Sigma_{i-1}(\beta))}{\det(\Sigma_{i-1}(\beta_0))}\right)^{\frac12}
\exp\left\{\frac{
1}{2\Delta_n}\sum_{l,m=1}^d\left(\overline{x}^l(\alpha)\Xi_{i-1}^{l,m}(\beta)\overline{x}^m(\alpha)-\overline{x}^l(\alpha_0)\Xi_{i-1}^{l,m}(\beta_0)\overline{x}^m(\alpha_0)\right)\right\}\\
&\quad \times\left[\frac{1}{\Delta_n}\sum_{m=1}^d\left(\Xi_{i-1}^{k,m}(\beta)\overline{x}^m(\alpha)-\Xi_{i-1}^{k,m}(\beta_0)\overline{x}^m(\alpha_0)\right)\right]^2\\
&\quad+\phi'\left(\left(\frac{\det(\Sigma_{i-1}(\beta))}{\det(\Sigma_{i-1}(\beta_0))}\right)^{\frac12}
\exp{\left\{\frac{
1}{2\Delta_n}\sum_{l,m=1}^d\left(\overline{x}^l(\alpha)\Xi_{i-1}^{l,m}(\beta)\overline{x}^m(\alpha)-\overline{x}^l(\alpha_0)\Xi_{i-1}^{l,m}(\beta_0)\overline{x}^m(\alpha_0)\right)\right\}}\right)\\
&\quad\times \left(\frac{\det(\Sigma_{i-1}(\beta))}{\det(\Sigma_{i-1}(\beta_0))}\right)^{\frac12}
\exp\left\{\frac{
1}{2\Delta_n}\sum_{l,m=1}^d\left(\overline{x}^l(\alpha)\Xi_{i-1}^{l,m}(\beta)\overline{x}^m(\alpha)-\overline{x}^l(\alpha_0)\Xi_{i-1}^{l,m}(\beta_0)\overline{x}^m(\alpha_0)\right)\right\}\\
&\quad\times \frac{1}{\Delta_n}(\Xi_{i-1}^{k,k'}(\beta)-\Xi_{i-1}^{k,k'}(\beta_0))
\end{align*}

By taking into account \eqref{eq:itotay}, one has that
{\small\begin{align*}
&E_{\theta_0}\left\{F(X_i)|\mathcal{G}_{i-1}^n\right\}\\
&=F(X_{i-1})+\Delta_n\sum_{k=1}^db_{i-1}^k(\alpha_0)\partial_{x_k}F(X_{i-1})+\frac{\Delta_n}{2}\sum_{k,k'=1}^d\Sigma_{i-1}^{k,k'}(\beta_0)\partial_{x_kx_{k'}}F(X_{i-1})+R(\theta,\Delta_n^2,X_{i-1})\\
&=\phi\left(\left(\frac{\det(\Sigma_{i-1}(\beta))}{\det(\Sigma_{i-1}(\beta_0))}\right)^{\frac12}
\exp\left\{\frac{
\Delta_n}{2}\sum_{l,m=1}^d\left(b_{i-1}^l(\alpha)\Xi_{i-1}^{l,m}(\beta)b_{i-1}^m(\alpha)-b_{i-1}^l(\alpha_0)\Xi_{i-1}^{l,m}(\beta_0)b_{i-1}^m(\alpha_0)\right)\right\}\right)\\
&\quad-\Delta_n\sum_{k=1}^db_{i-1}^k(\alpha_0)\Bigg\{\phi'\left(\left(\frac{\det(\Sigma_{i-1}(\beta))}{\det(\Sigma_{i-1}(\beta_0))}\right)^{\frac12}
\exp\left\{\frac{
\Delta_n}{2}\sum_{l,m=1}^d\left(b_{i-1}^l(\alpha)\Xi_{i-1}^{l,m}(\beta)b_{i-1}^m(\alpha)-b_{i-1}^l(\alpha_0)\Xi_{i-1}^{l,m}(\beta_0)b_{i-1}^m(\alpha_0)\right)\right\}\right)\\
&\quad\times \left(\frac{\det(\Sigma_{i-1}(\beta))}{\det(\Sigma_{i-1}(\beta_0))}\right)^{\frac12}
\exp\left\{\frac{
\Delta_n}{2}\sum_{l,m=1}^d\left(b_{i-1}^l(\alpha)\Xi_{i-1}^{l,m}(\beta)b_{i-1}^m(\alpha)-b_{i-1}^l(\alpha_0)\Xi_{i-1}^{l,m}(\beta_0)b_{i-1}^m(\alpha_0)\right)\right\}\\
&\quad \times\sum_{m=1}^d\left(\Xi_{i-1}^{k,m}(\beta)b_{i-1}^m(\alpha)-\Xi_{i-1}^{k,m}(\beta_0)b_{i-1}^m(\alpha_0)\right)\Bigg\}+\frac{\Delta_n}{2}\sum_{k,k'=1}^d\Sigma_{i-1}^{k,k'}(\beta_0)\\
&\quad\times\Bigg\{\phi''\left(\left(\frac{\det(\Sigma_{i-1}(\beta))}{\det(\Sigma_{i-1}(\beta_0))}\right)^{\frac12}
\exp\left\{\frac{
\Delta_n}{2}\sum_{l,m=1}^d\left(b_{i-1}^l(\alpha)\Xi_{i-1}^{l,m}(\beta)b_{i-1}^m(\alpha)-b_{i-1}^l(\alpha_0)\Xi_{i-1}^{l,m}(\beta_0)b_{i-1}^m(\alpha_0)\right)\right\}\right)\\
&\quad\times \left(\frac{\det(\Sigma_{i-1}(\beta))}{\det(\Sigma_{i-1}(\beta_0))}\right)
\exp\left\{
\Delta_n\sum_{l,m=1}^d\left(b_{i-1}^l(\alpha)\Xi_{i-1}^{l,m}(\beta)b_{i-1}^m(\alpha)-b_{i-1}^l(\alpha_0)\Xi_{i-1}^{l,m}(\beta_0)b_{i-1}^m(\alpha_0)\right)\right\}\\
&\quad \times\left[\sum_{m=1}^d\left(\Xi_{i-1}^{k,m}(\beta)b_{i-1}^m(\alpha)-\Xi_{i-1}^{k,m}(\beta_0)b_{i-1}^m(\alpha_0)\right)\right]^2\\
&\quad+\phi'\left(\left(\frac{\det(\Sigma_{i-1}(\beta))}{\det(\Sigma_{i-1}(\beta_0))}\right)^{\frac12}
\exp\left\{\frac{
\Delta_n}{2}\sum_{l,m=1}^d\left(b_{i-1}^l(\alpha)\Xi_{i-1}^{l,m}(\beta)b_{i-1}^m(\alpha)-b_{i-1}^l(\alpha_0)\Xi_{i-1}^{l,m}(\beta_0)b_{i-1}^m(\alpha_0)\right)\right\}\right)\\
&\quad\times \left(\frac{\det(\Sigma_{i-1}(\beta))}{\det(\Sigma_{i-1}(\beta_0))}\right)^{\frac12}
\exp\left\{\frac{
\Delta_n}{2}\sum_{l,m=1}^d\left(b_{i-1}^l(\alpha)\Xi_{i-1}^{l,m}(\beta)b_{i-1}^m(\alpha)-b_{i-1}^l(\alpha_0)\Xi_{i-1}^{l,m}(\beta_0)b_{i-1}^m(\alpha_0)\right)\right\}\\
&\quad \times\left[\sum_{m=1}^d\left(\Xi_{i-1}^{k,m}(\beta)b_{i-1}^m(\alpha)-\Xi_{i-1}^{k,m}(\beta_0)b_{i-1}^m(\alpha_0)\right)\right]^2\\
&\quad+\phi'\left(\left(\frac{\det(\Sigma_{i-1}(\beta))}{\det(\Sigma_{i-1}(\beta_0))}\right)^{\frac12}
\exp\left\{\frac{
\Delta_n}{2}\sum_{l,m=1}^d\left(b_{i-1}^l(\alpha)\Xi_{i-1}^{l,m}(\beta)b_{i-1}^m(\alpha)-b_{i-1}^l(\alpha_0)\Xi_{i-1}^{l,m}(\beta_0)b_{i-1}^m(\alpha_0)\right)\right\}\right)\\
&\quad\times \left(\frac{\det(\Sigma_{i-1}(\beta))}{\det(\Sigma_{i-1}(\beta_0))}\right)^{\frac12}
\exp\left\{\frac{
\Delta_n}{2}\sum_{l,m=1}^d\left(b_{i-1}^l(\alpha)\Xi_{i-1}^{l,m}(\beta)b_{i-1}^m(\alpha)-b_{i-1}^l(\alpha_0)\Xi_{i-1}^{l,m}(\beta_0)b_{i-1}^m(\alpha_0)\right)\right\}\\
&\quad\times \frac{1}{\Delta_n}(\Xi_{i-1}^{k,k'}(\beta)-\Xi_{i-1}^{k,k'}(\beta_0))\Bigg\}+R(\theta,\Delta_n^2,X_{i-1})
\end{align*}}

Therefore, the following result holds

{\small\begin{align}\label{eq:concon1}
&\frac1n\sum_{i=1}^nE_{\theta_0}\left\{F(X_i)|\mathcal{G}_{i-1}^n\right\}\notag\\
&=\frac1n\sum_{i=1}^n\Bigg\{\phi\left(\left(\frac{\det(\Sigma_{i-1}(\beta))}{\det(\Sigma_{i-1}(\beta_0))}\right)^{\frac12}
\exp\left\{\frac{
\Delta_n}{2}\sum_{l,m=1}^d\left(b_{i-1}^l(\alpha)\Xi_{i-1}^{l,m}(\beta)b_{i-1}^m(\alpha)-b_{i-1}^l(\alpha_0)\Xi_{i-1}^{l,m}(\beta_0)b_{i-1}^m(\alpha_0)\right)\right\}\right)\notag\\
&\quad+\phi'\left(\left(\frac{\det(\Sigma_{i-1}(\beta))}{\det(\Sigma_{i-1}(\beta_0))}\right)^{\frac12}
\exp\left\{\frac{
\Delta_n}{2}\sum_{l,m=1}^d\left(b_{i-1}^l(\alpha)\Xi_{i-1}^{l,m}(\beta)b_{i-1}^m(\alpha)-b_{i-1}^l(\alpha_0)\Xi_{i-1}^{l,m}(\beta_0)b_{i-1}^m(\alpha_0)\right)\right\}\right)\notag\\
&\quad\times \left(\frac{\det(\Sigma_{i-1}(\beta))}{\det(\Sigma_{i-1}(\beta_0))}\right)^{\frac12}
\exp\left\{\frac{
\Delta_n}{2}\sum_{l,m=1}^d\left(b_{i-1}^l(\alpha)\Xi_{i-1}^{l,m}(\beta)b_{i-1}^m(\alpha)-b_{i-1}^l(\alpha_0)\Xi_{i-1}^{l,m}(\beta_0)b_{i-1}^m(\alpha_0)\right)\right\}\notag\\
&\quad\times \frac{1}{2}\sum_{k,k'=1}^d\Sigma_{i-1}^{k,k'}(\beta_0)(\Xi_{i-1}^{k,k'}(\beta)-\Xi_{i-1}^{k,k'}(\beta_0))\Bigg\}+\frac1n\sum_{i=1}^nR(\theta,\Delta_n,X_{i-1})+\frac1n\sum_{i=1}^nR(\theta,\Delta_n^2,X_{i-1})\notag\\
&\stackrel{P_{\theta_0}}{\to}U_\phi(\beta,\beta_0)
\end{align}}
uniformly in $\theta$, where in the last step we have used Lemma \ref{lemmaer}.

Now, by means of the same arguments it is not hard to prove that
\begin{equation}\label{eq:concon2}
\frac{1}{n^2}\sum_{i=1}^nE_{\theta_0}\left\{F(X_i)^2|\mathcal{G}_{i-1}^n\right\}\stackrel{P_{\theta_0}}{\to}0
\end{equation}

In conclusion Lemma \ref{lemmaGJ}  and the results \eqref{eq:concon1}, \eqref{eq:concon2} implies the statement of the present Theorem.

\end{proof}

We point out that for $\phi(x)=\log(x)$, from Theorem \ref{conphi} we derive that
 $$\mathcal{D}_{\log,n}(\theta,\theta_0)\stackrel{P_{\theta_0}}{\to}\frac12\int\left[\mathrm{tr}(\Xi(\beta,x)\Sigma(\beta_0,x))-d+\log\left(\frac{\det(\Sigma(\beta,x))}{\det(\Sigma(\beta_0,x))}\right)\right]\mu_{\theta_0}(\de x)$$
which coincides with the result (4.1) in Kessler (1997).

Remark further that  $\mathcal{D}_{\phi,n}(\theta,\theta_0)$ can be used as a contrast function to derive minimum contrast estimators $\tilde \theta_n$ which solve $\mathcal{D}_{\phi,n}(\tilde \theta_n,\theta_0)=0$ whose properties can be studied using Theorem \ref{conphi}.

Hereafter we assume, as before,  that  $\phi:[0,\infty)\to \mathbb{R}$ is convex and continuous, its restriction on $(0,\infty)$ is finite, it is two times continuously differentiable and, in addition, that  $\phi(1)=\phi'(1)=0$ and $\phi''(1)=1$. 

Now introduce the family of test statistics defined as follows
$$
T_{\phi,n}(\theta,\theta_0):=2n\mathcal{D}_{\phi,n}(\theta,\theta_0).
$$
The above quantity is similar to that used by Morales {\it e al.} (1997), where $\mathcal{D}_{\phi,n}(\theta,\theta_0)$ is replaced by the true $\phi$-divergence.

Let $\hat\theta_n$ be the quasi maximum likelihood of the Bayes-type estimator defined in Section \ref{sec:preliminary}. The first step is to prove that the family of test statistics 
\begin{equation}\label{eq:test}T_{\phi,n}(\hat\theta_n,\theta_0)\end{equation}
is asymptotically distribution free under $H_0$.

\begin{theorem}\label{main}
Under $H_0$ and the conditions $\mathcal A1-\mathcal A7$, as $n\to \infty$,  we have
that
\begin{equation}\label{main2}
T_{\phi,n}(\hat\theta_n,\theta_0)\stackrel{d}{\to} \chi_{p+q}^2.
\end{equation}
\end{theorem}
\begin{proof}%[Proof of Theorem \ref{main}]

By Taylor's formula, we have that
\begin{align*}
&n\mathcal{D}_{\phi,n}(\hat\theta_n,\theta_0)\\
&=n\mathcal{D}_{\phi,n}(\theta_0,\theta_0)+n\partial_\theta\mathcal{D}_{\phi,n}(\theta_0,\theta_0)(\hat\theta_n-\theta_0)+\frac12(\hat\theta_n-\theta_0)'n\partial^2_\theta\mathcal{D}_{\phi,n}(\theta_0,\theta_0)(\hat\theta_n-\theta_0)+o_P(|\hat\theta_n-\theta_0|^2)\\
&=\frac12((\hat\theta_n-\theta_0)\varphi(n)^{-1/2})'\varphi(n)^{1/2}n\partial^2_\theta\mathcal{D}_{\phi,n}(\theta_0,\theta_0)\varphi(n)^{1/2}\varphi(n)^{-1/2}(\hat\theta_n-\theta_0)+o_P(1)
\end{align*}
where in the last step we have used the fact that $\phi(1)=\phi'(1)=0$.  We note that the $(j,k)$-element of the Hessian matrix $\partial^2_\theta\mathcal{D}_{\phi,n}(\theta,\theta_0)$ is given by
\begin{align*}
\partial^2_\theta\mathcal{D}^{j,k}_{\phi,n}(\theta,\theta_0)=\frac1n\sum_{i=1}^n\left\{\phi''\left(\frac{p_i(\theta)}{p_i(\theta_0)}\right)\frac{1}{p_i^2(\theta_0)}\partial_{\theta_j}p_i(\theta)\partial_{\theta_k}p_i(\theta)+\phi'\left(\frac{p_i(\theta)}{p_i(\theta_0)}\right)\frac{1}{p_i(\theta_0)}\partial^2_{\theta_j\theta_k} p_i(\theta)\right\}
\end{align*}
and then 
$$\partial^2_\theta\mathcal{D}^{j,k}_{\phi,n}(\theta_0,\theta_0)=\frac{1}{n}\sum_{i=1}^n\partial_{\theta_j}\mathbb{H}_i(\theta_0)\partial_{\theta_k}\mathbb{H}_i(\theta_0)=\frac1n\Lambda^{j,k}_n(\theta_0)$$

By taking into account Lemma \ref{lem1} and the convergence result  \eqref{eq:conest}, the statement of the Theorem follows immediately.

\end{proof}

Given the level $\alpha$, such test rejects $H_0$ if $T_{\phi,n}>c_{\alpha}$ where $c_{\alpha}$ is the $1-\alpha$ quantile of the limiting random variable $\chi_{p+q}^2$.
The power function of the proposed test  is equal to
$$\beta_\phi^n(\theta)= P_\theta\left\{T_{\phi,n}(\hat\theta_n,\theta_0)>c_{\alpha}\right\},\qquad \theta\neq\theta_0.
$$
 This power function can be studied under the contiguous alternative setup. Indeed, in this case
 we are able to approximate $\beta_\phi^n(\theta)$ by means of a distribution function of a non-central chi square random variable.
 \begin{theorem}\label{main3}
Under the conditions $\mathcal A1-\mathcal A7$, $H_0 : \theta=\theta_0$ and the alternative contiguous hypotheses $H_1:\theta=\theta_0+\varphi(n)h$, where $h\in \{h\in \mathbb{R}^{p+q}:\theta=\theta_0+\varphi(n)h \in \Theta\}$,
we have that
\begin{equation}\label{main3}
\beta_\phi^n(\theta)\cong
1-\mathbf{F}_{p+q}\left(c_{\alpha}\right),
\end{equation}
where $\mathbf{F}_{p+q}(\cdot)$ is the cumulative function
of the random variable $\chi^2_{p+q}(\mu)$ which is a non-central chi square random variable with $p+q$ degrees of freedom and noncentrality parameter $\mu=h'\mathcal{I}(\theta_0)h$.
\end{theorem}
\begin{proof}%[Proof of Theorem \ref{main3}]
Under $H_1:\theta=\theta_0+\varphi(n)^{\frac12}h$ we have that 
\begin{eqnarray}\label{pow1}
\varphi(n)^{-\frac12}(\hat\theta_n-\theta_0)=\varphi(n)^{-\frac12}(\hat\theta_n-\theta)+h
\stackrel{d}{\to}N(h,\mathcal{I}(\theta_0)^{-1})
\end{eqnarray}
where in the last step we have taken into account the convergence in \eqref{eq:conest} and that $\sup_n|\mathcal{I}(\theta)-\mathcal{I}(\theta_0)|=o_P(1)$ which is implied by the continuity of $\mathcal{I}(\theta)$ in $\theta_0$. Therefore, by means of the same arguments exploited in the proof of Theorem \ref{main}, we can write that
\begin{align*}
T_{\phi,n}(\hat\theta_n,\theta_0)&=((\hat\theta_n-\theta_0)\varphi(n)^{-1/2})'\varphi(n)^{1/2}n\partial^2_\theta\mathcal{D}_{\phi,n}(\theta_0,\theta_0)\varphi(n)^{1/2}\varphi(n)^{-1/2}(\hat\theta_n-\theta_0)+o_P(1)\\
&\stackrel{d}{\to}\chi^2_{p+q}(\mu)
\end{align*}
This last step concludes the proof.

\end{proof}

It is well known that  the likelihood ratio test is the uniformly most powerful test for all sample sizes for testing $\theta_0$ against any simple alternative $\theta_1$. 
In the framework of this paper, the natural benchmark test statistics is the generalised quasi-likelihood ratio test (GQLRT), that is
$$S_{n}(\hat\theta_n,\theta_0)=2n\mathcal{D}_{\log,n}(\hat\theta_n,\theta_0)=2[\mathbb{H}_n({\bf X}_n,\hat\theta_n)-\mathbb{H}_n({\bf X}_n,\theta_0)]$$
 (i.e. when $\theta_0$ is estimated using the maximum likelihood estimator).
We observe that $S_{n}(\hat\theta_n,\theta_0)$ is not a member of the family of tests \eqref{eq:test}, because it could be obtained from \eqref{eq:test} for $\phi(x) = \log(x)$ which does not satisfy the requirement of this Section. Nevertheless, the limiting distribution of GQLRTs is the same of $T_{\phi,n}(\hat\theta_n,\theta_0)$ as it is easy to verify, namely
$$S_{n}(\hat\theta_n,\theta_0)\stackrel{d}{\to}\chi^2_{p+q}$$
This implies, as expected, that all test statistics $T_{\phi,n}(\hat\theta_n,\theta_0)$ are asymptotically equivalent to the $S_{n}(\hat\theta_n,\theta_0)$. Next section investigates the behaviour for small sample sizes.

\section{Numerical analysis}\label{sec:numerics}

Although all  test statistics $T_{\phi,n}$ of Section \ref{sec:pseudo} satisfy  the
same asymptotic results of Theorems \ref{main2} and \ref{main3}, for small sample sizes the performance of the test
is determined by the statistical model generating the data and the sample size.
In our numerical study we will consider the power of the test under local alternatives as in Theorem \ref{main3} for different $\phi$ functions, sample sizes of $n=50, 100, 250, 500, 1000$ observations and $T=n^\frac13$, in order to satisfy the asymptotic theory. For testing $\theta_0$ against the local alternatives $\theta_0 + \frac{h}{\sqrt{n\Delta_n}}$ for the parameters in the drift coefficient and $\theta_0 + \frac{h}{\sqrt{n}}$ for the parameters in the diffusion coefficient, $h$ is taken in a grid from $0$ to $1$, and $h=0$ corresponds to the null hypothesis $H_0$.
We consider the following $\phi$ functions, which are all such that $\phi(1)=\phi'(1)=0$ and $\phi''(1)=1$.
\begin{itemize}
\item $\phi(x) = 1-x+x \log(x)$: the true $\phi$-divergence  based on this function is equivalent to the Kullback-Leibler divergence, but in our setup this is not true. We use the label $AKL$ in the table for this approximate KL;
\item $\phi(x) = \frac{x^{\lambda+1} -x - \lambda(x-1)}{\lambda(\lambda+1)}, \lambda\neq -1,0$: this corresponds to the power divergence studied in Cressie and Read (1984). We use $\lambda=k$ in the tables, with $k=-30, -10, -3$;
\item $\phi(x) = \left(\frac{x-1}{x+1}\right)^2$: this was proposed in the Balakrishnan and  Sanghvi (1968), we name it BS in the tables.
\end{itemize}
For the data generating process, we consider the following statistical models
\begin{itemize}
\item[OU:]  the one-dimensional Ornstein-Uhlenbeck model solution to  $\de X_t = (\alpha-\beta X_t)\de t + \sigma \de W_t$, $X_0=1$, with $\theta_0=(\alpha, \beta, \sigma) = (0.5, 0.5, 0.25)$;
\item[GBM:]  the one-dimensional geometric Brownian motion model solution to  $\de X_t = (\alpha-\beta X_t)\de t + \sigma X_t \de W_t$, $X_0=1$, with $\theta_0=(\alpha, \beta, \sigma) = (0.5, 0.5, 0.25)$;
\item[CIR:]  the one-dimensional CIR model solution to  $\de X_t = (\alpha-\beta X_t)\de t + \sigma \sqrt{X_t} \de W_t$, $X_0=1$, with $\theta_0=(\alpha, \beta, \sigma) = (0.5, 0.5, 0.125)$;
\item[MOU:]  the two-dimensional Ornstein-Uhlenbeck model  solution to  $\de X_t^{(1)} = (2-\mu_1X_t^{(1)})\de t + \sigma_1 \de W_t^{(1)}$, $\de X_t^{(2)} = (2-\mu_2X_t^{(2)})\de t + \sigma_2 \de W_t^{(2)}$,
 $X_0=(1,1)$, with $\theta_0=(\mu_1,\mu_2,\sigma_1, \sigma_2) = (1, 1, 0.3, 0.5)$;
\end{itemize}
In all experiments, the process have been simulated at high frequency using the Euler-Maruyama scheme and resampled to obtain $n=50, 100, 250, 500, 1000$ observations.
The $T_{\phi,n}$ statistics is constructed using the quasi maximum likelihood estimator.
Each experiment is replicated 1000 times and the Tables \ref{tab:OU}, \ref{tab:GBM}, \ref{tab:CIR}, and \ref{tab:MOU} contain the empirical power function, i.e. the average values   
$$
\beta_h = \frac{\#\{T_{n,\phi}(\hat\theta_n, \theta_0+h/\varphi_n) > \tilde c_{\alpha}\}}{1000}
$$
and the empirical level of the test
$$
\alpha = \frac{\#\{T_{n,\phi}(\hat\theta_n, \theta_0) > \tilde c_{\alpha}\}}{1000}
$$
where $\tilde c_\alpha$ is the estimated $\alpha$ quantile of the empirical distribution of $T_{\phi,n}$, with $\alpha=0.05$.
The choice of using the empirical threshold $\tilde c_\alpha$ instead of the theoretical threshold $c_\alpha$ from the $\chi^2_d$ distribution, is due to the fact that otherwise the test are non comparable, i.e the level of test are not $\alpha$ and, for example, when $h=0$ the test for different $\phi$'s produces different empirical level of the test.

In the Tables \ref{tab:OU}, \ref{tab:GBM}, \ref{tab:CIR}, and \ref{tab:MOU} we have used the bold face font to put in evidence the test statistics with the highest empirical power function $\beta_h$ for a given local alternative $h>0$.
As mentioned before,  the natural benchmark test statistics is the generalised quasi likelihood ratio test GQLRT.

From the numerical analysis we can see several facts
\begin{itemize} 
\item the test statistic based on the AKL $\phi$ function is not equivalent to the GQLRT as in the case of true $\phi$-divergence test statistics;
\item the GQLRT test statistics appears to be (almost) uniformly more powerful for small $h$ when the sample size is small ($n=50$) but, as soon as the sample size increases, other divergences test statistics $T_{\phi,n}$ have higher power;
\item  in all cases, the GQLRT is not most powerful test for all alternatives and most of the times, the power divergences are more powerful;
\item there seems to be no uniformly most powerful test in among the set of $\phi$-test statistics proposed here.
\end{itemize}

In conclusion, this paper proposed a new family of test statistics for discretely observed diffusion processes. Although closely related to the $\phi$-divergence test statistics, the test proposed here is different and new.
The empirical analysis shows that, despite the asymptotic equivalence of all test statistics proposed here, for small sample size and different statistical models, there is no uniformly most powerful test among the members of this family. Furthermore, the generalised quasi likelihood test
statistics, which does not belong to this class, is not necessarily the optimal test as well.

\section{Extension to other classes of diffusion processes}\label{ext}

The methodology developed in this paper is quite general and under suitable conditions can be applied for testing other ergodic diffusion models observed at discrete times. Therefore, in this Section we discuss, briefly, some possible extensions of the test statistics \eqref{eq:test} without delving into the technical aspects.  

 Let us start considering a generalization of \eqref{eq:sde} given by an ergodic multidimensional diffusion process with jumps, that is
\begin{equation}\label{eq:sdejumps}
\de X_t = b(\alpha, X_{t-}) \de t + \sigma(\beta,X_{t-})  \de W_t+\int_Ec(X_{t-},z,\alpha)p(dt,dz),\quad X_0=x_0,
\end{equation}
where $p(dt,dz)$ is a Poisson random measure on $\mathbb{R}^+\times E, E=\mathbb{R}^d-\{0\}$ and $q^\alpha(\de t,\de z)$ is its intensity measure, that is $E(p(\de t,\de z))=q^{\alpha}(\de t,\de z)$. We set $q^{\alpha}(\de t,\de z)=f_\alpha(z)\de z\de t$ and $f_\alpha(z)=\lambda(\alpha)F_{\alpha}(z)$ where $\lambda(\alpha)$ is a nonnegative function and $F_{\alpha}(z)$ represents a probability density. The parametric estimation of \eqref{eq:sdejumps} has been tackled by Shimizu and Yoshida (2006) and Ogihara and Yoshida (2011) with a similar but slightly different contrast function. Shimizu (2006) dealt with $M$-estimators for the same statistical problem. Now, we give some sketches on the contrast function used in the previous papers to which we refer for major details about the parametric estimate of the stochastic differential equation model \eqref{eq:sdejumps}.  In particular, Ogihara and Yoshida (2011) introduced the following quasi-likelihood function
\begin{align}\label{cfjumps}
&\mathbb{H}_n({\bf X}_n,\theta)\notag\\
&:=-\frac12\sum_{i=1}^n\left\{\log\det(\Sigma_{i-1}(\beta))
+\frac{
1}{\Delta_n}\overline{X}_i'(\alpha)\Xi_{i-1}(\beta)\overline{X}_i(\alpha)\right\}1_{\{|X_{t_i}-X_{t_{i-1}}|\leq D\Delta_n^\rho\}}\notag\\
&\quad +\sum_{i=1}^n\Big\{\log\left[\Psi_{\alpha}(X_{i-1},X_{t_i}-X_{t_{i-1}})\varphi_n(X_{i-1},X_{t_i}-X_{t_{i-1}})\right]\varphi_n(X_{i-1},\Delta X_i)1_{\{|X_{t_i}-X_{t_{i-1}}|\leq D\Delta_n^\rho\}}\notag\\
&\quad-{\Delta_n}\int \Psi_{\alpha}(X_{i-1},y)\varphi_n(X_{i-1},y)\de y\Big\}
\end{align}
where $D>0$, $\rho$ is a suitable constant, $\Psi(y,x)=f_\theta(c^{-1}(x,y,\theta))|J(x,y,\theta)|$ where $J(x,y,\theta)$ is the Jacobian of $c^{-1}(x,y,\theta)$ and $0\leq \varphi_n(x,y)\leq 1$ is a sequence of real valued functions satisfying condition [H10] in Ogihara and Yoshida (2011). 

The introduced contrast function is very natural since it is split in two parts: the first component is the contrast for an usual diffusion process, and the second one emerges from the discretization of the likelihood function of an compound Poisson process with Levy density $f_\alpha$ . Therefore, if $|X_{t_i}-X_{t_{i-1}}|\leq D\Delta_n^\rho$ the function $\mathbb{H}_n({\bf X}_n,\theta)$ judges no jumps occur in the interval $(t_{i-1},t_i]$ and it reduces to the quasi-likelihood function for a diffusion process without jumps; otherwise  $\mathbb{H}_n({\bf X}_n,\theta)$ judges a jumps occurs in the previous time interval. 

By taking into account a rapidly increasing scheme, under suitable assumptions, Ogihara and Yoshida (2011) proved that the quasi-maximum likelihood estimator and a Bayes type estimator based on \eqref{cfjumps} are consistent and asymptotically gaussian. Furthermore, it is possible to prove that Lemma \ref{lem1} holds. This implies that the test statistics $T_{\phi,n}$, opportunely modified for testing \eqref{eq:sdejumps}, weakly converges to a chi-squared random variable.

Another random model which attracted the attention of the researchers has been the small-diffusion process. In this case we consider the following stochastic differential equation
\begin{equation}\label{eq:sdesmall}
\de X_t = b(\alpha, X_t) \de t +\varepsilon \sigma(\beta,X_t)  \de W_t,\quad X_0=x_0,
\end{equation}
where $t\in[0,1]$ and $\varepsilon\in(0,1]$. S\o rensen and Uchida (2003) and Gloter and S\o rensen (2009), introduced the contrast function

\begin{align}\label{cfsmall}
&\mathbb{H}_{\varepsilon ,n}({\bf X}_n,\theta):=\sum_{i=1}^n\left\{\log\det(\Sigma_{i-1}(\beta))
+\varepsilon^{-2}n\overline{X}_i'(\alpha)\Xi_{i-1}(\beta)\overline{X}_i(\alpha)\right\}
\end{align}
where $\overline{X}_i(\alpha)=X_{t_i}-X_{t_{i-1}}-\frac1n b(\alpha,X_{t_i})$. In this framework $\varepsilon\to 0$ as $n\to\infty$. Under suitable conditions, it is possible to show that $\hat\theta_{\varepsilon ,n}=(\hat\alpha_{\varepsilon ,n},\hat\beta_{\varepsilon ,n})=\min_\theta\mathbb{H}_{\varepsilon ,n}({\bf X}_n,\theta)$ is consistent and 
$$\binom{\varepsilon^{-1}(\hat\alpha_{\varepsilon ,n}-\alpha_0)}{\sqrt{n}(\hat\beta_{\varepsilon ,n}-\beta_0)}\stackrel{d}{\to} N(0,\mathcal I(\theta_0)^{-1})$$
where $\mathcal I(\theta)$ is the Fisher information matrix for \eqref{eq:sdesmall}. Furthermore, Lemma \ref{lem1} admits a version based on the small diffusion process \eqref{eq:sdesmall}. Then the test statistics \eqref{eq:test} constructed by means of \eqref{cfsmall} have the asymptotic properties proved in Theorem \ref{main2}-\ref{main3}.  

%\section{Tables}

\begin{table}[ht]
\begin{center}
{\scriptsize 
\begin{tabular}{c c}
$n=50$ & $n=100$\\
\begin{tabular}{rrrrrrr}
  \hline
 & AKL & GQLRT & BS & $\lambda=-20$ & $\lambda=-10$ & $\lambda=-3$ \\ 
  \hline
h=0.00 & 0.050 & 0.050 & 0.050 & 0.050 & 0.050 & 0.050 \\ 
  h=0.01 & 0.047 & \bf 0.054 & 0.045 & 0.053 & 0.053 & 0.050 \\ 
  h=0.05 & 0.031 & \bf  0.069 & 0.044 & 0.064 & 0.067 & 0.048 \\ 
  h=0.10 & 0.019 &  \bf 0.095 & 0.047 & 0.085 & 0.088 & 0.065 \\ 
  h=0.20 & 0.013 &  \bf 0.172 & 0.072 & 0.152 & 0.159 & 0.103 \\ 
  h=0.30 & 0.018 &  \bf 0.291 & 0.144 & 0.247 & 0.258 & 0.184 \\ 
  h=0.40 & 0.028 &  \bf 0.443 & 0.246 & 0.347 & 0.362 & 0.310 \\ 
  h=0.50 & 0.056 &  \bf 0.616 & 0.398 & 0.441 & 0.462 & 0.437 \\ 
  h=0.60 & 0.115 & \bf  0.765 & 0.530 & 0.535 & 0.551 & 0.559 \\ 
  h=0.70 & 0.194 &  \bf 0.873 & 0.677 & 0.604 & 0.617 & 0.671 \\ 
  h=0.80 & 0.324 &  \bf 0.929 & 0.802 & 0.657 & 0.679 & 0.775 \\ 
  h=0.90 & 0.448 &  \bf 0.971 & 0.898 & 0.705 & 0.720 & 0.870 \\ 
  h=1.00 & 0.594 &  \bf 0.987 & 0.947 & 0.734 & 0.749 & 0.922 \\ 
   \hline
\end{tabular}
&
\begin{tabular}{rrrrrrr}
  \hline
 & AKL & GQLRT & BS & $\lambda=-20$ & $\lambda=-10$ & $\lambda=-3$ \\ 
  \hline
h=0.00 & 0.050 & 0.050 & 0.050 & 0.050 & 0.050 & 0.050 \\ 
  h=0.01 & 0.046 & 0.054 & 0.051 &\bf 0.058 &\bf  0.058 & 0.052 \\ 
  h=0.05 & 0.028 & 0.061 & 0.045 & 0.083 &\bf  0.086 & 0.054 \\ 
  h=0.10 & 0.022 & 0.087 & 0.047 & 0.134 &\bf  0.141 & 0.068 \\ 
  h=0.20 & 0.009 & 0.168 & 0.074 & 0.241 &\bf  0.253 & 0.144 \\ 
  h=0.30 & 0.020 & 0.281 & 0.164 & 0.359 &\bf  0.374 & 0.253 \\ 
  h=0.40 & 0.029 & 0.463 & 0.282 & 0.501 &\bf  0.525 & 0.404 \\ 
  h=0.50 & 0.072 & 0.633 & 0.443 & 0.624 &\bf  0.641 & 0.544 \\ 
  h=0.60 & 0.164 & \bf 0.774 & 0.613 & 0.725 & 0.753 & 0.698 \\ 
  h=0.70 & 0.272 & \bf 0.880 & 0.749 & 0.814 & 0.843 & 0.819 \\ 
  h=0.80 & 0.421 & \bf 0.943 & 0.856 & 0.877 & 0.907 & 0.899 \\ 
  h=0.90 & 0.594 & \bf 0.974 & 0.922 & 0.921 & 0.944 & 0.952 \\ 
  h=1.00 & 0.729 & \bf 0.989 & 0.966 & 0.951 & 0.968 & 0.984 \\ 
   \hline
\end{tabular}\\
\\
$n=250$ & $n=500$\\
\begin{tabular}{rrrrrrr}
  \hline
 & AKL & GQLRT & BS & $\lambda=-20$ & $\lambda=-10$ & $\lambda=-3$ \\ 
  \hline
h=0.00 & 0.050 & 0.050 & 0.050 & 0.050 & 0.050 & 0.050 \\ 
  h=0.01 & 0.044 & 0.054 & 0.047 & 0.053 &\bf  0.059 & 0.052 \\ 
  h=0.05 & 0.035 & 0.063 & 0.049 & 0.072 &\bf   0.080 & 0.063 \\ 
  h=0.10 & 0.025 & 0.091 & 0.054 & 0.107 &\bf   0.118 & 0.076 \\ 
  h=0.20 & 0.036 & 0.176 & 0.091 & 0.199 & \bf  0.221 & 0.152 \\ 
  h=0.30 & 0.061 & 0.301 & 0.200 & 0.323 & \bf  0.376 & 0.290 \\ 
  h=0.40 & 0.136 & 0.481 & 0.353 & 0.480 &\bf   0.539 & 0.471 \\ 
  h=0.50 & 0.276 & 0.678 & 0.537 & 0.643 &\bf   0.719 & 0.648 \\ 
  h=0.60 & 0.434 & 0.825 & 0.721 & 0.782 & \bf  0.837 & 0.813 \\ 
  h=0.70 & 0.632 & \bf  0.925 & 0.866 & 0.878 & 0.923 & 0.917 \\ 
  h=0.80 & 0.787 &\bf   0.977 & 0.945 & 0.931 & 0.972 & 0.966 \\ 
  h=0.90 & 0.906 & \bf  0.989 & 0.985 & 0.975 & 0.988 & 0.993 \\ 
  h=1.00 & 0.960 & 0.998 & 0.995 & 0.992 &\bf   0.999 & 0.998 \\ 
   \hline
\end{tabular}
&
\begin{tabular}{rrrrrrr}
  \hline
 & AKL & GQLRT & BS & $\lambda=-20$ & $\lambda=-10$ & $\lambda=-3$ \\ 
  \hline
h=0.00 & 0.050 & 0.050 & 0.050 & 0.050 & 0.050 & 0.050 \\ 
  h=0.01 & 0.048 & 0.050 & 0.048 &\bf  0.055 &\bf   0.055 & 0.054 \\ 
  h=0.05 & 0.035 & 0.058 & 0.046 & 0.075 &\bf   0.077 & 0.061 \\ 
  h=0.10 & 0.027 & 0.071 & 0.059 & 0.104 &\bf   0.117 & 0.075 \\ 
  h=0.20 & 0.051 & 0.158 & 0.118 & 0.222 & \bf  0.247 & 0.172 \\ 
  h=0.30 & 0.109 & 0.293 & 0.243 & 0.369 &\bf   0.405 & 0.314 \\ 
  h=0.40 & 0.235 & 0.473 & 0.421 & 0.560 & \bf  0.615 & 0.501 \\ 
  h=0.50 & 0.408 & 0.684 & 0.631 & 0.721 &\bf   0.778 & 0.698 \\ 
  h=0.60 & 0.615 & 0.828 & 0.794 & 0.854 &\bf   0.902 & 0.836 \\ 
  h=0.70 & 0.793 & 0.928 & 0.910 & 0.938 &\bf   0.958 & 0.938 \\ 
  h=0.80 & 0.903 & 0.970 & 0.965 & 0.968 & \bf  0.983 & 0.977 \\ 
  h=0.90 & 0.963 & 0.992 & 0.991 & 0.992 & \bf  0.998 & 0.993 \\ 
  h=1.00 & 0.990 & 0.999 & 0.998 & 0.999 & \bf  1.000 &\bf   1.000 \\ 
   \hline
\end{tabular}\\
\\
\end{tabular}
$n=1000$\\
\begin{tabular}{rrrrrrr}
  \hline
 & AKL & GQLRT & BS & $\lambda=-20$ & $\lambda=-10$ & $\lambda=-3$ \\ 
  \hline
h=0.00 & 0.050 & 0.050 & 0.050 & 0.050 & 0.050 & 0.050 \\ 
  h=0.01 & 0.047 & 0.049 & 0.051 & \bf 0.057 & 0.055 & 0.050 \\ 
  h=0.05 & 0.038 & 0.057 & 0.050 & \bf  0.089 & 0.076 & 0.054 \\ 
  h=0.10 & 0.047 & 0.078 & 0.065 & 0.120 & 0 \bf .122 & 0.071 \\ 
  h=0.20 & 0.076 & 0.138 & 0.108 &  \bf 0.243 & 0.232 & 0.134 \\ 
  h=0.30 & 0.143 & 0.263 & 0.230 &  \bf 0.430 & 0.429 & 0.259 \\ 
  h=0.40 & 0.288 & 0.444 & 0.411 & 0.630 &  \bf 0.640 & 0.454 \\ 
  h=0.50 & 0.494 & 0.635 & 0.615 & 0.798 &  \bf 0.800 & 0.668 \\ 
  h=0.60 & 0.692 & 0.805 & 0.795 & 0.908 & \bf  0.920 & 0.836 \\ 
  h=0.70 & 0.848 & 0.925 & 0.917 & 0.970 &  \bf 0.973 & 0.933 \\ 
  h=0.80 & 0.943 & 0.977 & 0.974 & 0.995 &  \bf 0.996 & 0.979 \\ 
  h=0.90 & 0.982 & 0.995 & 0.996 &  \bf 0.999 & \bf  0.999 & 0.997 \\ 
  h=1.00 & 0.997 & \bf  1.000 &  \bf 1.000 & \bf  1.000 & \bf  1.000 & \bf  1.000 \\ 
   \hline
\end{tabular}
}
\end{center}
\caption{Empirical power function $\beta_h$, for different sample sizes $n$ and local alternatives $h$. The empirical power and theoretical power is $\alpha=0.05$. Data generating model: the 1-dimensional Ornstein-Uhlenbeck process.}
\label{tab:OU}
\end{table}

\begin{table}[ht]
\begin{center}
{\scriptsize 
\begin{tabular}{c c}
$n=50$ & $n=100$\\
\begin{tabular}{rrrrrrr}
  \hline
 & AKL & GQLRT & BS & $\lambda=-20$ & $\lambda=-10$ & $\lambda=-3$ \\ 
  \hline
h=0.00 & 0.050 & 0.050 & 0.050 & 0.050 & 0.050 & 0.050 \\ 
  h=0.01 & 0.046 & \bf 0.057 & 0.048 & 0.053 & 0.053 & 0.052 \\ 
  h=0.05 & 0.033 &  \bf 0.073 & 0.040 & 0.067 & 0.065 & 0.054 \\ 
  h=0.10 & 0.017 &  \bf 0.098 & 0.048 & 0.092 & 0.092 & 0.065 \\ 
  h=0.20 & 0.007 & \bf  0.175 & 0.062 & 0.167 & 0.164 & 0.105 \\ 
  h=0.30 & 0.011 &  \bf 0.296 & 0.127 & 0.260 & 0.259 & 0.194 \\ 
  h=0.40 & 0.018 & \bf  0.444 & 0.237 & 0.364 & 0.368 & 0.312 \\ 
  h=0.50 & 0.042 & \bf  0.607 & 0.373 & 0.458 & 0.463 & 0.433 \\ 
  h=0.60 & 0.085 & \bf  0.752 & 0.502 & 0.543 & 0.545 & 0.544 \\ 
  h=0.70 & 0.154 &  \bf 0.867 & 0.657 & 0.604 & 0.616 & 0.667 \\ 
  h=0.80 & 0.241 &  \bf 0.928 & 0.774 & 0.663 & 0.671 & 0.775 \\ 
  h=0.90 & 0.385 &  \bf 0.966 & 0.867 & 0.719 & 0.727 & 0.850 \\ 
  h=1.00 & 0.508 & 0 \bf .986 & 0.935 & 0.754 & 0.761 & 0.915 \\ 
   \hline
\end{tabular}
&
\begin{tabular}{rrrrrrr}
  \hline
 & AKL & GQLRT & BS & $\lambda=-20$ & $\lambda=-10$ & $\lambda=-3$ \\ 
  \hline
h=0.00 & 0.050 & 0.050 & 0.050 & 0.050 & 0.050 & 0.050 \\ 
  h=0.01 & 0.046 & 0.056 & 0.046 & \bf 0.058 & 0.057 & 0.046 \\ 
  h=0.05 & 0.027 & 0.057 & 0.040 &\bf  0.081 & 0.079 & 0.057 \\ 
  h=0.10 & 0.019 & 0.084 & 0.046 & 0.120 & \bf 0.124 & 0.073 \\ 
  h=0.20 & 0.014 & 0.174 & 0.071 & 0.230 & \bf 0.232 & 0.142 \\ 
  h=0.30 & 0.019 & 0.289 & 0.152 & 0.345 & \bf 0.357 & 0.253 \\ 
  h=0.40 & 0.029 & 0.462 & 0.271 & 0.477 & \bf 0.493 & 0.401 \\ 
  h=0.50 & 0.071 &\bf  0.634 & 0.431 & 0.600 & 0.627 & 0.553 \\ 
  h=0.60 & 0.159 &\bf  0.780 & 0.583 & 0.716 & 0.740 & 0.700 \\ 
  h=0.70 & 0.273 &\bf  0.885 & 0.734 & 0.804 & 0.829 & 0.825 \\ 
  h=0.80 & 0.418 &\bf  0.940 & 0.842 & 0.866 & 0.891 & 0.901 \\ 
  h=0.90 & 0.585 & \bf 0.973 & 0.917 & 0.910 & 0.933 & 0.951 \\ 
  h=1.00 & 0.720 & \bf 0.988 & 0.960 & 0.944 & 0.959 & 0.980 \\ 
   \hline
\end{tabular}
\\
\\
$n=250$ & $n=500$\\
\begin{tabular}{rrrrrrr}
  \hline
 & AKL & GQLRT & BS & $\lambda=-20$ & $\lambda=-10$ & $\lambda=-3$ \\ 
  \hline
h=0.00 & 0.050 & 0.050 & 0.050 & 0.050 & 0.050 & 0.050 \\ 
  h=0.01 & 0.045 & 0.053 & 0.049 & \bf 0.058 & 0.055 & 0.051 \\ 
  h=0.05 & 0.036 & 0.063 & 0.044 & \bf  0.084 & 0.078 & 0.060 \\ 
  h=0.10 & 0.027 & 0.091 & 0.050 & \bf  0.109 & 0.108 & 0.074 \\ 
  h=0.20 & 0.025 & 0.170 & 0.094 & 0.205 &  \bf 0.208 & 0.148 \\ 
  h=0.30 & 0.051 & 0.301 & 0.193 & 0.334 &  \bf 0.350 & 0.285 \\ 
  h=0.40 & 0.118 & 0.483 & 0.340 & 0.498 &  \bf 0.527 & 0.460 \\ 
  h=0.50 & 0.244 & 0.674 & 0.530 & 0.659 &  \bf 0.693 & 0.643 \\ 
  h=0.60 & 0.405 &  \bf 0.825 & 0.717 & 0.790 & 0.814 & 0.808 \\ 
  h=0.70 & 0.598 &  \bf 0.924 & 0.866 & 0.879 & 0.910 & 0.907 \\ 
  h=0.80 & 0.765 &  \bf 0.973 & 0.941 & 0.937 & 0.962 & 0.966 \\ 
  h=0.90 & 0.885 & 0.989 & 0.979 & 0.975 & 0.988 &  \bf 0.991 \\ 
  h=1.00 & 0.954 & \bf  0.999 & 0.995 & 0.993 & 0.998 & 0.998 \\ 
   \hline
\end{tabular}
&
\begin{tabular}{rrrrrrr}
  \hline
 & AKL & GQLRT & BS & $\lambda=-20$ & $\lambda=-10$ & $\lambda=-3$ \\ 
  \hline
h=0.00 & 0.050 & 0.050 & 0.050 & 0.050 & 0.050 & 0.050 \\ 
  h=0.01 & 0.048 & 0.047 & 0.049 &\bf  0.056 & 0.055 & 0.046 \\ 
  h=0.05 & 0.039 & 0.055 & 0.053 & 0.076 & \bf  0.078 & 0.053 \\ 
  h=0.10 & 0.026 & 0.079 & 0.061 & 0.114 &\bf   0.119 & 0.068 \\ 
  h=0.20 & 0.053 & 0.160 & 0.127 & 0.229 & \bf  0.244 & 0.159 \\ 
  h=0.30 & 0.118 & 0.296 & 0.253 & 0.373 &\bf   0.421 & 0.300 \\ 
  h=0.40 & 0.243 & 0.481 & 0.446 & 0.576 & \bf  0.623 & 0.494 \\ 
  h=0.50 & 0.436 & 0.695 & 0.650 & 0.723 & \bf  0.770 & 0.689 \\ 
  h=0.60 & 0.636 & 0.831 & 0.800 & 0.857 & \bf  0.909 & 0.825 \\ 
  h=0.70 & 0.794 & 0.929 & 0.919 & 0.936 & \bf  0.963 & 0.933 \\ 
  h=0.80 & 0.908 & 0.969 & 0.969 & 0.970 & \bf  0.984 & 0.977 \\ 
  h=0.90 & 0.966 & 0.992 & 0.992 & 0.994 &\bf   0.998 & 0.992 \\ 
  h=1.00 & 0.990 & 0.998 & 0.998 & 0.999 &\bf   1.000 &\bf   1.000 \\ 
   \hline
\end{tabular}\\
\\
\end{tabular}
\\
$n=1000$\\
\begin{tabular}{rrrrrrr}
  \hline
 & AKL & GQLRT & BS & $\lambda=-20$ & $\lambda=-10$ & $\lambda=-3$ \\ 
  \hline
h=0.00 & 0.050 & 0.050 & 0.050 & 0.050 & 0.050 & 0.050 \\ 
  h=0.01 & 0.047 & 0.052 & 0.045 & \bf 0.063 & 0.054 & 0.047 \\ 
  h=0.05 & 0.033 & 0.056 & 0.043 & \bf 0.090 & 0.078 & 0.055 \\ 
  h=0.10 & 0.039 & 0.083 & 0.060 & \bf 0.125 &\bf  0.125 & 0.071 \\ 
  h=0.20 & 0.068 & 0.139 & 0.104 &\bf  0.248 & 0.246 & 0.134 \\ 
  h=0.30 & 0.139 & 0.259 & 0.215 & 0.437 & \bf 0.440 & 0.267 \\ 
  h=0.40 & 0.283 & 0.448 & 0.397 & 0.630 & \bf 0.642 & 0.456 \\ 
  h=0.50 & 0.474 & 0.647 & 0.595 & 0.807 & \bf 0.817 & 0.677 \\ 
  h=0.60 & 0.686 & 0.821 & 0.787 & 0.915 & \bf 0.924 & 0.840 \\ 
  h=0.70 & 0.845 & 0.928 & 0.909 & 0.978 & \bf 0.976 & 0.934 \\ 
  h=0.80 & 0.938 & 0.976 & 0.971 & 0.993 & \bf 0.997 & 0.982 \\ 
  h=0.90 & 0.981 & 0.995 & 0.997 & \bf 0.999 & \bf 0.999 & 0.997 \\ 
  h=1.00 & 0.997 &\bf  1.000 &\bf  1.000 &\bf  1.000 &\bf  1.000 &\bf  1.000 \\ 
   \hline
\end{tabular}
}
\end{center}
\caption{Empirical power function $\beta_h$, for different sample sizes $n$ and local alternatives $h$. The empirical power and theoretical power is $\alpha=0.05$. Data generating model: the 1-dimensional geometric Brownian Motion process.}
\label{tab:GBM}
\end{table}

\begin{table}[ht]
\begin{center}
{\scriptsize 
\begin{tabular}{c c}
$n=50$ & $n=100$\\
\begin{tabular}{rrrrrrr}
  \hline
 & AKL & GQLRT & BS & $\lambda=-20$ & $\lambda=-10$ & $\lambda=-3$ \\ 
  \hline
h=0.00 & 0.050 & 0.050 & 0.050 & 0.050 & 0.050 & 0.050 \\ 
  h=0.01 & 0.038 & 0.054 & 0.046 & 0.054 & \bf 0.056 & 0.050 \\ 
  h=0.05 & 0.014 &  \bf 0.093 & 0.055 & 0.083 & 0.088 & 0.062 \\ 
  h=0.10 & 0.009 &  \bf 0.167 & 0.076 & 0.145 & 0.152 & 0.095 \\ 
  h=0.20 & 0.014 &  \bf 0.434 & 0.252 & 0.327 & 0.342 & 0.291 \\ 
  h=0.30 & 0.072 &  \bf 0.747 & 0.528 & 0.504 & 0.524 & 0.521 \\ 
  h=0.40 & 0.213 &  \bf 0.927 & 0.799 & 0.634 & 0.650 & 0.746 \\ 
  h=0.50 & 0.471 &  \bf 0.986 & 0.944 & 0.723 & 0.737 & 0.902 \\ 
  h=0.60 & 0.720 &  \bf 0.998 & 0.989 & 0.759 & 0.771 & 0.969 \\ 
  h=0.70 & 0.902 &  \bf 1.000 & 0.998 & 0.783 & 0.808 & 0.995 \\ 
  h=0.80 & 0.971 &  \bf 1.000 & \bf  1.000 & 0.809 & 0.826 & 0.999 \\ 
  h=0.90 & 0.995 &  \bf 1.000 &  \bf 1.000 & 0.814 & 0.836 & 0.999 \\ 
  h=1.00 & 0.998 &  \bf 1.000 & \bf  1.000 & 0.814 & 0.836 &  \bf 1.000 \\ 
   \hline
\end{tabular}
&
\begin{tabular}{rrrrrrr}
  \hline
 & AKL & GQLRT & BS & $\lambda=-20$ & $\lambda=-10$ & $\lambda=-3$ \\ 
  \hline
h=0.00 & 0.050 & 0.050 & 0.050 & 0.050 & 0.050 & 0.050 \\ 
  h=0.01 & 0.040 & 0.053 & 0.049 & 0.063 &\bf  0.065 & 0.050 \\ 
  h=0.05 & 0.017 & 0.082 & 0.043 & 0.119 & \bf  0.133 & 0.062 \\ 
  h=0.10 & 0.007 & 0.166 & 0.065 & 0.235 & \bf  0.241 & 0.131 \\ 
  h=0.20 & 0.025 & 0.465 & 0.250 & 0.480 & \bf  0.505 & 0.386 \\ 
  h=0.30 & 0.130 & \bf  0.772 & 0.557 & 0.704 & 0.729 & 0.678 \\ 
  h=0.40 & 0.379 & \bf  0.941 & 0.833 & 0.857 & 0.889 & 0.893 \\ 
  h=0.50 & 0.683 &\bf   0.987 & 0.956 & 0.940 & 0.956 & 0.982 \\ 
  h=0.60 & 0.887 & \bf  1.000 & 0.991 & 0.966 & 0.980 & 0.996 \\ 
  h=0.70 & 0.977 & \bf  1.000 & 0.999 & 0.980 & 0.990 & 0.999 \\ 
  h=0.80 & 0.993 & \bf  1.000 & \bf  1.000 & 0.987 & 0.991 & \bf  1.000 \\ 
  h=0.90 & 1.000 & \bf  1.000 & \bf  1.000 & 0.990 & 0.994 & \bf  1.000 \\ 
  h=1.00 & 1.000 & \bf  1.000 & \bf  1.000 & 0.991 & 0.994 & \bf  1.000 \\ 
    \hline
\end{tabular}
\\
\\
$n=250$ & $n=500$\\
\begin{tabular}{rrrrrrr}
  \hline
 & AKL & GQLRT & BS & $\lambda=-20$ & $\lambda=-10$ & $\lambda=-3$ \\ 
  \hline
h=0.00 & 0.050 & 0.050 & 0.050 & 0.050 & 0.050 & 0.050 \\ 
  h=0.01 & 0.040 & 0.052 & 0.047 & 0.060 & \bf 0.063 & 0.051 \\ 
  h=0.05 & 0.022 & 0.081 & 0.054 & 0.110 & \bf  0.111 & 0.076 \\ 
  h=0.10 & 0.017 & 0.165 & 0.093 & 0.203 &  \bf 0.212 & 0.135 \\ 
  h=0.20 & 0.103 & 0.464 & 0.337 & 0.475 & \bf  0.514 & 0.438 \\ 
  h=0.30 & 0.371 & 0.812 & 0.714 & 0.771 &  \bf 0.817 & 0.793 \\ 
  h=0.40 & 0.739 &  \bf 0.971 & 0.937 & 0.932 & 0.963 & 0.962 \\ 
  h=0.50 & 0.941 &  \bf 0.999 & 0.995 & 0.991 &  \bf 0.999 & 0.997 \\ 
  h=0.60 & 0.996 &  \bf 1.000 & \bf  1.000 &  \bf 1.000 &  \bf 1.000 &  \bf 1.000 \\ 
  h=0.70 &  \bf 1.000 &  \bf 1.000 &  \bf 1.000 &  \bf 1.000 &  \bf 1.000 &  \bf 1.000 \\ 
  h=0.80 &  \bf 1.000 &  \bf 1.000 &  \bf 1.000 &  \bf 1.000 &  \bf 1.000 &  \bf 1.000 \\ 
  h=0.90 &  \bf 1.000 &  \bf 1.000 &  \bf 1.000 &  \bf 1.000 &  \bf 1.000 &  \bf 1.000 \\ 
  h=1.00 &  \bf 1.000 &  \bf 1.000 &  \bf 1.000 &  \bf 1.000 &  \bf 1.000 &  \bf 1.000 \\ 
    \hline
\end{tabular}
&
\begin{tabular}{rrrrrrr}
  \hline
 & AKL & GQLRT & BS & $\lambda=-20$ & $\lambda=-10$ & $\lambda=-3$ \\ 
  \hline
 h=0.00 & 0.050 & 0.050 & 0.050 & 0.050 & 0.050 & 0.050 \\ 
  h=0.01 & 0.043 & 0.046 & 0.051 &\bf  0.061 & 0.060 & 0.054 \\ 
  h=0.05 & 0.027 & 0.073 & 0.063 & 0.113 &\bf   0.122 & 0.072 \\ 
  h=0.10 & 0.043 & 0.151 & 0.117 & 0.228 & \bf  0.256 & 0.165 \\ 
  h=0.20 & 0.216 & 0.463 & 0.419 & 0.566 & \bf  0.612 & 0.477 \\ 
  h=0.30 & 0.600 & 0.819 & 0.788 & 0.856 & \bf  0.909 & 0.825 \\ 
  h=0.40 & 0.894 & 0.966 & 0.965 & 0.964 & \bf  0.982 & 0.975 \\ 
  h=0.50 & 0.988 & 0.995 & 0.995 & \bf  0.999 & \bf  0.999 &\bf   0.999 \\ 
  h=0.60 & 0.999 & \bf  1.000 &\bf   1.000 & \bf  1.000 &\bf   1.000 &\bf   1.000 \\ 
  h=0.70 & \bf  1.000 & \bf  1.000 & \bf  1.000 & \bf  1.000 & \bf  1.000 & \bf  1.000 \\ 
  h=0.80 & \bf  1.000 & \bf  1.000 & \bf  1.000 & \bf  1.000 & \bf  1.000 & \bf  1.000 \\ 
  h=0.90 & \bf  1.000 & \bf  1.000 & \bf  1.000 & \bf  1.000 & \bf  1.000 & \bf  1.000 \\ 
  h=1.00 & \bf  1.000 & \bf  1.000 & \bf  1.000 & \bf  1.000 & \bf  1.000 & \bf  1.000 \\ 
    \hline
\end{tabular}\\
\\
\end{tabular}\\
$n=1000$\\
\begin{tabular}{rrrrrrr}
  \hline
 & AKL & GQLRT & BS & $\lambda=-20$ & $\lambda=-10$ & $\lambda=-3$ \\ 
  \hline
h=0.00 & 0.050 & 0.050 & 0.050 & 0.050 & 0.050 & 0.050 \\ 
  h=0.01 & 0.042 & 0.050 & 0.047 & \bf 0.065 & 0.056 & 0.050 \\ 
  h=0.05 & 0.040 & 0.081 & 0.062 & 0.123 &\bf  0.125 & 0.070 \\ 
  h=0.10 & 0.065 & 0.146 & 0.103 & \bf 0.241 & 0.238 & 0.135 \\ 
  h=0.20 & 0.266 & 0.439 & 0.395 & 0.624 & \bf 0.638 & 0.445 \\ 
  h=0.30 & 0.666 & 0.805 & 0.776 & 0.906 & \bf 0.913 & 0.832 \\ 
  h=0.40 & 0.928 & 0.971 & 0.970 & 0.992 & \bf 0.994 & 0.974 \\ 
  h=0.50 & 0.995 & \bf 1.000 & \bf 1.000 & \bf 1.000 & \bf 1.000 & \bf 1.000 \\ 
  h=0.60 & \bf 1.000 & \bf 1.000 & \bf 1.000 & \bf 1.000 & \bf 1.000 & \bf 1.000 \\ 
  h=0.70 & \bf 1.000 & \bf 1.000 & \bf 1.000 & \bf 1.000 & \bf 1.000 & \bf 1.000 \\ 
  h=0.80 & \bf 1.000 & \bf 1.000 & \bf 1.000 & \bf 1.000 & \bf 1.000 & \bf 1.000 \\ 
  h=0.90 & \bf 1.000 & \bf 1.000 & \bf 1.000 & \bf 1.000 & \bf 1.000 & \bf 1.000 \\ 
  h=1.00 & \bf 1.000 & \bf 1.000 & \bf 1.000 & \bf 1.000 & \bf 1.000 & \bf 1.000 \\ 
   \hline
\end{tabular}
}
\end{center}
\caption{Empirical power function $\beta_h$, for different sample sizes $n$ and local alternatives $h$. The empirical power and theoretical power is $\alpha=0.05$. Data generating model: the 1-dimensional CIR process.}
\label{tab:CIR}
\end{table}

\begin{table}[ht]
\begin{center}
{\scriptsize 
\begin{tabular}{c c}
$n=50$ & $n=100$\\
\begin{tabular}{rrrrrrr}
  \hline
 & AKL & GQLRT & BS & $\lambda=-20$ & $\lambda=-10$ & $\lambda=-3$ \\ 
  \hline
h=0.00 & 0.050 & 0.050 & 0.050 & 0.050 & 0.050 & 0.050 \\ 
  h=0.01 & 0.049 & 0.052 & 0.050 & 0.053 & \bf 0.054 & 0.052 \\ 
  h=0.05 & 0.043 & \bf 0.061 & 0.053 & 0.059 &\bf  0.061 & 0.060 \\ 
  h=0.10 & 0.037 & \bf 0.084 & 0.063 & 0.071 & 0.074 & 0.072 \\ 
  h=0.20 & 0.031 & \bf 0.165 & 0.101 & 0.098 & 0.099 & 0.111 \\ 
  h=0.30 & 0.045 &\bf  0.310 & 0.194 & 0.146 & 0.147 & 0.176 \\ 
  h=0.40 & 0.077 & \bf 0.492 & 0.333 & 0.191 & 0.191 & 0.268 \\ 
  h=0.50 & 0.146 & \bf 0.695 & 0.493 & 0.241 & 0.243 & 0.388 \\ 
  h=0.60 & 0.274 & \bf 0.860 & 0.680 & 0.282 & 0.284 & 0.508 \\ 
  h=0.70 & 0.439 &\bf  0.940 & 0.806 & 0.331 & 0.336 & 0.635 \\ 
  h=0.80 & 0.620 &\bf  0.985 & 0.904 & 0.380 & 0.383 & 0.740 \\ 
  h=0.90 & 0.775 & \bf 0.998 & 0.962 & 0.418 & 0.421 & 0.828 \\ 
  h=1.00 & 0.890 & \bf 0.999 & 0.989 & 0.445 & 0.452 & 0.901 \\ 
   \hline
\end{tabular}
&
\begin{tabular}{rrrrrrr}
  \hline
 & AKL & GQLRT & BS & $\lambda=-20$ & $\lambda=-10$ & $\lambda=-3$ \\ 
  \hline
h=0.00 & 0.050 & 0.050 & 0.050 & 0.050 & 0.050 & 0.050 \\ 
  h=0.01 & 0.049 & 0.052 & 0.052 & 0.051 & 0.052 & \bf 0.054 \\ 
  h=0.05 & 0.041 & \bf  0.065 & 0.054 & 0.059 & 0.060 & 0.063 \\ 
  h=0.10 & 0.038 & \bf  0.091 & 0.063 & 0.078 & 0.076 & 0.068 \\ 
  h=0.20 & 0.039 & \bf  0.186 & 0.106 & 0.124 & 0.127 & 0.115 \\ 
  h=0.30 & 0.058 & \bf  0.354 & 0.213 & 0.185 & 0.190 & 0.220 \\ 
  h=0.40 & 0.106 & \bf  0.568 & 0.377 & 0.275 & 0.286 & 0.360 \\ 
  h=0.50 & 0.232 & \bf  0.774 & 0.573 & 0.371 & 0.379 & 0.528 \\ 
  h=0.60 & 0.399 & \bf  0.914 & 0.762 & 0.473 & 0.487 & 0.704 \\ 
  h=0.70 & 0.596 & \bf  0.972 & 0.894 & 0.566 & 0.581 & 0.831 \\ 
  h=0.80 & 0.803 & \bf  0.996 & 0.957 & 0.650 & 0.667 & 0.919 \\ 
  h=0.90 & 0.905 & \bf  0.998 & 0.985 & 0.713 & 0.732 & 0.964 \\ 
  h=1.00 & 0.963 &  \bf 1.000 & 0.996 & 0.758 & 0.775 & 0.987 \\ 
     \hline
\end{tabular}
\\
\\
$n=250$ & $n=500$\\
\begin{tabular}{rrrrrrr}
  \hline
 & AKL & GQLRT & BS & $\lambda=-20$ & $\lambda=-10$ & $\lambda=-3$ \\ 
  \hline
h=0.00 & 0.050 & 0.050 & 0.050 & 0.050 & 0.050 & 0.050 \\ 
  h=0.01 & 0.048 & 0.054 & 0.052 & 0.055 & \bf 0.059 & 0.054 \\ 
  h=0.05 & 0.042 & 0.065 & 0.057 & \bf 0.075 & 0.074 & 0.063 \\ 
  h=0.10 & 0.046 & 0.082 & 0.073 & 0.104 & \bf 0.105 & 0.077 \\ 
  h=0.20 & 0.074 & \bf 0.197 & 0.136 & 0.161 & 0.165 & 0.155 \\ 
  h=0.30 & 0.152 & \bf 0.384 & 0.283 & 0.255 & 0.274 & 0.312 \\ 
  h=0.40 & 0.294 & \bf 0.597 & 0.467 & 0.390 & 0.414 & 0.500 \\ 
  h=0.50 & 0.492 & \bf 0.800 & 0.681 & 0.540 & 0.584 & 0.706 \\ 
  h=0.60 & 0.707 & \bf 0.922 & 0.846 & 0.674 & 0.733 & 0.862 \\ 
  h=0.70 & 0.866 & \bf 0.976 & 0.942 & 0.811 & 0.862 & 0.950 \\ 
  h=0.80 & 0.949 & \bf 0.995 & 0.983 & 0.901 & 0.940 & 0.985 \\ 
  h=0.90 & 0.988 & \bf 1.000 & 0.995 & 0.954 & 0.975 & 0.997 \\ 
  h=1.00 & 0.995 & \bf 1.000 & 0.999 & 0.976 & 0.989 &\bf  1.000 \\ 
  \hline
\end{tabular}
&
\begin{tabular}{rrrrrrr}
  \hline
 & AKL & GQLRT & BS & $\lambda=-20$ & $\lambda=-10$ & $\lambda=-3$ \\ 
  \hline
  h=0.00 & 0.050 & 0.050 & 0.050 & 0.050 & 0.050 & 0.050 \\ 
  h=0.01 & 0.050 & 0.053 & 0.053 & 0.056 & 0.053 & \bf 0.057 \\ 
  h=0.05 & 0.050 & 0.065 & 0.067 & \bf  0.073 &  0.070 & 0.064 \\ 
  h=0.10 & 0.062 & 0.098 & 0.084 & 0.097 &  0.098 &  \bf 0.100 \\ 
  h=0.20 & 0.117 & 0.190 & 0.165 & 0.179 &  \bf 0.195 & 0.184 \\ 
  h=0.30 & 0.229 & 0.359 & 0.324 & 0.313 & \bf  0.361 & \bf  0.361 \\ 
  h=0.40 & 0.423 &  \bf 0.596 & 0.547 & 0.478 & 0.542 & 0.582 \\ 
  h=0.50 & 0.651 &  \bf 0.790 & 0.742 & 0.660 & 0.731 & 0.770 \\ 
  h=0.60 & 0.836 &  \bf 0.930 & 0.902 & 0.805 & 0.875 & 0.914 \\ 
  h=0.70 & 0.940 &  \bf 0.979 & 0.970 & 0.916 & 0.961 & 0.975 \\ 
  h=0.80 & 0.986 &  \bf 0.997 & 0.995 & 0.977 & 0.990 & 0.995 \\ 
  h=0.90 & 0.998 &  0.999 & 0.998 & 0.992 & \bf  1.000 & 0.999 \\ 
  h=1.00 & 0.999 &  \bf 1.000 &  \bf 1.000 &  \bf 1.000 & \bf  1.000 & \bf  1.000 \\ 
    \hline
\end{tabular}\\
\\
\end{tabular}
$n=1000$\\
\begin{tabular}{rrrrrrr}
  \hline
 & AKL & GQLRT & BS & $\lambda=-20$ & $\lambda=-10$ & $\lambda=-3$ \\ 
  \hline
  h=0.00 & 0.050 & 0.050 & 0.050 & 0.050 & 0.050 & 0.050 \\ 
  h=0.01 & 0.045 & 0.049 & 0.050 & \bf 0.056 & 0.053 & 0.052 \\ 
  h=0.05 & 0.039 & 0.052 & 0.057 &\bf  0.071 & 0.068 & 0.058 \\ 
  h=0.10 & 0.044 & 0.074 & 0.064 & \bf 0.113 & 0.094 & 0.073 \\ 
  h=0.20 & 0.072 & 0.140 & 0.125 & 0.222 &\bf  0.237 & 0.149 \\ 
  h=0.30 & 0.163 & 0.294 & 0.266 & 0.408 & \bf 0.430 & 0.304 \\ 
  h=0.40 & 0.331 & 0.549 & 0.511 & 0.608 & \bf 0.653 & 0.560 \\ 
  h=0.50 & 0.592 & 0.745 & 0.721 & 0.788 & \bf 0.833 & 0.756 \\ 
  h=0.60 & 0.791 & 0.903 & 0.883 & 0.934 & \bf 0.950 & 0.909 \\ 
  h=0.70 & 0.920 & 0.967 & 0.957 & 0.976 & \bf 0.984 & 0.969 \\ 
  h=0.80 & 0.974 &\bf  0.994 & 0.991 & 0.992 & \bf 0.994 & \bf 0.994 \\ 
  h=0.90 & 0.994 & \bf 0.997 & \bf 0.997 & \bf 0.997 & \bf 0.997 & \bf 0.997 \\ 
  h=1.00 & 0.998 & 0.999 & 0.999 & 0.999 & \bf 1.000 & 0.999 \\ 
   \hline
\end{tabular}
}
\end{center}
\caption{Empirical power function $\beta_h$, for different sample sizes $n$ and local alternatives $h$. The empirical power and theoretical power is $\alpha=0.05$. Data generating model: the 2-dimensional Ornstein-Uhlenbeck process.}
\label{tab:MOU}

\end{table}
\section{Appendix}

We collect in the following some results which assume a crucial role in the proofs of the present work.

\begin{lemma}\label{lemmagr}
For $k\geq 1$ and $t_i\geq t\geq t_{i-1}$
$$E_{\theta_0}\{|X_t-X_{t_{i-1}}|^k|\mathcal{G}_{i-1}^n\}\leq C_k|t-t_{i-1}|(1+|X_{t_{i-1}}|)^{C_k}$$
\end{lemma}
\begin{proof}
It is an application of the Gronwall-Belman lemma (see the proof of Lemma 6 in Kessler, 1997).
\end{proof}

\begin{lemma}\label{lem0}
Let $\overline{X}_i=X_{t_i}-X_{t_{i-1}}-\Delta_n b(\alpha_0,X_{t_i})$. Under the assumptions $\mathcal{A}1$-$\mathcal{A}7$, for $k_j=1,2,...,d$ and $j=1,2,3,4$, we have that
\begin{align*}
&E_{\theta_0}\{\overline{X}_i^{k_1}|\mathcal{G}_{i-1}^n\}=R(\theta_0,\Delta_n^2,X_{i-1})\\
&E_{\theta_0}\{\overline{X}_i^{k_1}\overline{X}_i^{k_2}|\mathcal{G}_{i-1}^n\}=\Delta_n\Sigma_{i-1}^{k_1,k_2}(\beta_0)+R(\theta_0,\Delta_n^2,X_{i-1}),\\
&E_{\theta_0}\{\overline{X}_i^{k_1}\overline{X}_i^{k_2}\overline{X}_i^{k_3}|\mathcal{G}_{i-1}^n\}=R(\theta_0,\Delta_n^2,X_{i-1}),\\
&E_{\theta_0}\left\{\prod_{j=1}^4\overline{X}_i^{k_j}|\mathcal{G}_{i-1}^n\right\}=\Delta_n^2(\Sigma_{i-1}^{k_1,k_2}(\beta_0)\Sigma_{i-1}^{k_3,k_4}(\beta_0)+\Sigma_{i-1}^{k_1,k_3}(\beta_0)\Sigma_{i-1}^{k_2,k_4}(\beta_0)+\Sigma_{i-1}^{k_1,k_4}(\beta_0)\Sigma_{i-1}^{k_2,k_3}(\beta_0))\\
&\quad\quad\quad\qquad \qquad \qquad \quad +R(\theta_0,\Delta_n^3,X_{i-1})
\end{align*}
\end{lemma}
\begin{proof}
Crucial for the proof is the following result.
If $f\in C^{2(l+1)}$, by applying Ito's formula repeatedly, we can prove that
\begin{equation}\label{eq:itotay}
E_{\theta}\{f(X_i)|\mathcal{G}_{i-1}^n\}=\sum_{k=0}^l\frac{\Delta_n^k}{k!}\mathcal{L}_\theta^kf(X_{i-1})+R(\theta,\Delta_n^{l+1},X_{i-1})
\end{equation}
where
$$\mathcal{L}_\theta f(x)=\sum_{i=1}^db^i(\alpha,x)\partial_{x_i}f(x)+\frac12\sum_{i,j=1}^d\Sigma^{i,j}(\beta,x)\partial_{x_ix_j}^2f(x)$$
represents the infinitesimal generator for $X_t,t\in[0,T]$ and $\mathcal{L}_\theta^k$ is the $k$th interate of $\mathcal{L}_\theta$. Furthermore $\mathcal{L}_\theta^0$ is the identity function. 

We only prove the last equality in the statement of Lemma. The others equalities follow by similar arguments and require a less number of calculations. 

Let $f(x_1,x_2,x_3,x_4)=\prod_{j=1}^4(x_j-X_{i-1}^{k_j}-\Delta_n b_{i-1}^{k_j}(\alpha_0))$, from \eqref{eq:itotay} for $l=2$, we obtain that

\begin{align}\label{eq:itotay2}
\left.E_{\theta_0}\left\{\prod_{j=1}^4\overline{X}_i^{k_j}\right|\mathcal{G}_{i-1}^n\right\}&=E_{\theta_0}\{f(\overline{X}_i^{k_1},\overline{X}_i^{k_2},\overline{X}_i^{k_3},\overline{X}_i^{k_4})|\mathcal{G}_{i-1}^n\}\notag\\
&=\Delta_n^4\prod_{j=1}^4b_{i-1}^{k_j}(\alpha_0)+\Delta_n\mathcal{L}_{\theta_0}^1f\left(\overline{X}_{i-1}^{k_1},\overline{X}_{i-1}^{k_2},\overline{X}_{i-1}^{k_3},\overline{X}_{i-1}^{k_4}\right)\notag\\
&\quad+\frac{\Delta_n^2}{2}\mathcal{L}_{\theta_0}^2f\left(\overline{X}_{i-1}^{k_1},\overline{X}_{i-1}^{k_2},\overline{X}_{i-1}^{k_3},\overline{X}_{i-1}^{k_4}\right)+R(\theta_0,\Delta_n^{3},X_{i-1})
\end{align}
Therefore, we can write that
\begin{align*}
&\mathcal{L}_{\theta_0}^1f(x_1,x_2,x_3,x_4)\\
&=b_{i-1}^{k_1}(\alpha_0)(x_2-X_{i-1}^{k_2}-\Delta_n b_{i-1}^{k_2}(\alpha_0))(x_3-X_{i-1}^{k_3}-\Delta_n b_{i-1}^{k_3}(\alpha_0))(x_4-X_{i-1}^{k_4}-\Delta_n b_{i-1}^{k_4}(\alpha_0))\\
&\quad+b_{i-1}^{k_2}(\alpha_0)(x_1-X_{i-1}^{k_1}-\Delta_n b_{i-1}^{k_1}(\alpha_0))(x_3-X_{i-1}^{k_3}-\Delta_n b_{i-1}^{k_3}(\alpha_0))(x_4-X_{i-1}^{k_4}-\Delta_n b_{i-1}^{k_4}(\alpha_0))\\
&\quad+b_{i-1}^{k_3}(\alpha_0)(x_1-X_{i-1}^{k_1}-\Delta_n b_{i-1}^{k_1}(\alpha_0))(x_2-X_{i-1}^{k_2}-\Delta_n b_{i-1}^{k_2}(\alpha_0))(x_4-X_{i-1}^{k_4}-\Delta_n b_{i-1}^{k_4}(\alpha_0))\\
&\quad+b_{i-1}^{k_4}(\alpha_0)(x_1-X_{i-1}^{k_1}-\Delta_n b_{i-1}^{k_1}(\alpha_0))(x_2-X_{i-1}^{k_2}-\Delta_n b_{i-1}^{k_2}(\alpha_0))(x_3-X_{i-1}^{k_3}-\Delta_n b_{i-1}^{k_3}(\alpha_0))\\
&\quad+\frac12\Bigg\{\Sigma_{i-1}^{k_1,k_2}(\beta_0)(x_3-X_{i-1}^{k_3}-\Delta_n b_{i-1}^{k_3}(\alpha_0))(x_4-X_{i-1}^{k_4}-\Delta_n b_{i-1}^{k_4}(\alpha_0))\\
&\quad+\Sigma_{i-1}^{k_1,k_3}(\beta_0)(x_2-X_{i-1}^{k_2}-\Delta_n b_{i-1}^{k_2}(\alpha_0))(x_4-X_{i-1}^{k_4}-\Delta_n b_{i-1}^{k_4}(\alpha_0))\\
&\quad+\Sigma_{i-1}^{k_1,k_4}(\beta_0)(x_2-X_{i-1}^{k_2}-\Delta_n b_{i-1}^{k_2}(\alpha_0))(x_3-X_{i-1}^{k_3}-\Delta_n b_{i-1}^{k_3}(\alpha_0))\\
&\quad+\Sigma_{i-1}^{k_2,k_3}(\beta_0)(x_1-X_{i-1}^{k_1}-\Delta_n b_{i-1}^{k_1}(\alpha_0))(x_4-X_{i-1}^{k_4}-\Delta_n b_{i-1}^{k_4}(\alpha_0))\\
&\quad+\Sigma_{i-1}^{k_2,k_4}(\beta_0)(x_1-X_{i-1}^{k_1}-\Delta_n b_{i-1}^{k_1}(\alpha_0))(x_3-X_{i-1}^{k_3}-\Delta_n b_{i-1}^{k_3}(\alpha_0))\\
&\quad+\Sigma_{i-1}^{k_3,k_4}(\beta_0)(x_1-X_{i-1}^{k_1}-\Delta_n b_{i-1}^{k_1}(\alpha_0))(x_2-X_{i-1}^{k_2}-\Delta_n b_{i-1}^{k_2}(\alpha_0))\Bigg\}
\end{align*}
and thus
\begin{align}\label{eq:ig1}
\Delta_n\mathcal{L}_{\theta_0}^1f(\overline{X}_{i-1}^{k_1},\overline{X}_{i-1}^{k_2},\overline{X}_{i-1}^{k_3},\overline{X}_{i-1}^{k_4})&=-4\Delta_n^4\prod_{j=1}^4b_{i-1}^{k_j}(\alpha_0)+\frac{\Delta_n^3}{2}\{\Sigma_{i-1}^{k_1,k_2}(\beta_0)b_{i-1}^{k_3}(\alpha_0)b_{i-1}^{k_4}(\alpha_0)\notag\\
&\quad+\Sigma_{i-1}^{k_1,k_3}(\beta_0)b_{i-1}^{k_2}(\alpha_0)b_{i-1}^{k_4}(\alpha_0)+\Sigma_{i-1}^{k_1,k_4}(\beta_0)b_{i-1}^{k_2}(\alpha_0)b_{i-1}^{k_3}(\alpha_0)\notag\\
&\quad +\Sigma_{i-1}^{k_2,k_3}(\beta_0)b_{i-1}^{k_1}(\alpha_0)b_{i-1}^{k_4}(\alpha_0)+\Sigma_{i-1}^{k_2,k_4}(\beta_0)b_{i-1}^{k_1}(\alpha_0)b_{i-1}^{k_3}(\alpha_0)\notag\\
&\quad+\Sigma_{i-1}^{k_3,k_4}(\beta_0)b_{i-1}^{k_1}(\alpha_0)b_{i-1}^{k_2}(\alpha_0)\}\notag\\
&=R(\theta_0,\Delta_n^{3},X_{i-1})+R(\theta_0,\Delta_n^{4},X_{i-1})
\end{align}
Similar and cumbersome calculations lead to the following equality
\begin{align}\label{eq:ig2}
&\frac{\Delta_n^2}{2}\mathcal{L}_{\theta_0}^2f(\overline{X}_{i-1}^{k_1},\overline{X}_{i-1}^{k_2},\overline{X}_{i-1}^{k_3},\overline{X}_{i-1}^{k_4})\notag\\
&=\frac{\Delta_n^2}{2}\Bigg[\sum_{j=1}^4b_{i-1}^{k_j}(\alpha_0)\partial_{x_j}\mathcal{L}_{\theta_0}^1f(\overline{X}_{i-1}^{k_1},\overline{X}_{i-1}^{k_2},\overline{X}_{i-1}^{k_3},\overline{X}_{i-1}^{k_4})\notag\\
&\quad+\frac12\sum_{h,j=1}^4\Sigma_{i-1}^{k_h,k_j}(\beta_0)\partial_{x_hx_j}^2\mathcal{L}_{\theta_0}^1f(\overline{X}_{i-1}^{k_1},\overline{X}_{i-1}^{k_2},\overline{X}_{i-1}^{k_3},\overline{X}_{i-1}^{k_4})\Bigg]\notag\\
&=\Delta_n^2(\Sigma_{i-1}^{k_1,k_2}(\beta_0)\Sigma_{i-1}^{k_3,k_4}(\beta_0)+\Sigma_{i-1}^{k_1,k_3}(\beta_0)\Sigma_{i-1}^{k_2,k_4}(\beta_0)+\Sigma_{i-1}^{k_1,k_4}(\beta_0)\Sigma_{i-1}^{k_2,k_3}(\beta_0))\notag\\
&\quad+R(\theta_0,\Delta_n^{3},X_{i-1})+R(\theta_0,\Delta_n^{4},X_{i-1})
\end{align}
Therefore, by taking into account the equalities \eqref{eq:ig1} and \eqref{eq:ig2} the expansion \eqref{eq:itotay2} leads to the result present in the statement of Lemma.

\end{proof}

\begin{lemma}\label{lemmaer}
Let $f:\mathbb{R}^d\times\Theta\to \mathbb{R}$ be such that $f$ is differentiable with respect to $x$ and $\theta$, with derivatives of polynomial growth in $x$ uniformly in $\theta$. Under the assumptions $\mathcal{A}1$-$\mathcal{A}7$, we have that
$$\frac1n\sum_{i=1}^n f(X_{i-1},\theta)\stackrel{P_{\theta_0}}{\to}\int f(x,\theta)\mu_{\theta_0}(\de x)$$
uniformly in $\theta$.
\end{lemma}
\begin{proof} See the proof of Lemma 8 in Kessler (1997).
\end{proof}

\begin{lemma}\label{lemmaGJ}
If $U_i,i=1,2,...,n$ are random variables $\mathcal{G}_{i}$-measurable, then the two following conditions imply $\sum_{i=1}^nU_i\stackrel{P}{\to}U$:
 \begin{align*}
 &\sum_{i=1}^nE\{U_i|\mathcal{G}_{i-1}\}\stackrel{P}{\to}U\\
 &\sum_{i=1}^nE\{U_i^2|\mathcal{G}_{i-1}\}\stackrel{P}{\to}0
 \end{align*}
\end{lemma}
\begin{proof} See Lemma 9 in Genon-Catalot and Jacod (1993).
\end{proof}

\small{

}

 \end{document}